\documentclass{amsart}
\usepackage{amssymb}
\usepackage{amsmath}
\usepackage[export]{adjustbox}
\usepackage{graphicx}
\usepackage[all,cmtip]{xy}

\input{diagxy}

\usepackage[colorlinks=true]{hyperref}

\author{O. Braunling, M. Groechenig, A. Heleodoro, J. Wolfson}

\thanks{O.B.\ was supported by DFG GK 1821 \textquotedblleft Cohomological Methods in
Geometry\textquotedblright and a Junior Fellowship at the Freiburg Institute for Advanced Studies (FRIAS). M.G.\ was partially supported by EPSRC Grant
No.\ EP/G06170X/1 and has received funding from the European Union's Horizon 2020 research and innovation programme under the Marie Sk\l odowska-Curie Grant Agreement No. 701679. J.W.\ was partially supported by an NSF Post-doctoral
Research Fellowship under Grant No.\ DMS-1400349. Our research was supported
in part by NSF Grant No.\ DMS-1303100 and EPSRC Mathematics Platform grant EP/I019111/1.\\ \includegraphics[height=1cm,right]{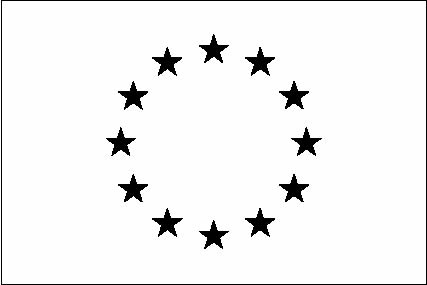}
}

\address{Freiburg Institute for Advanced Studies (FRIAS),  University of
	Freiburg}
\email{oliver.braeunling@math.uni-freiburg.de}
\address{Department of Mathematics, University of Toronto}
\email{michael.groechenig@utoronto.ca}
\address{Department of Mathematics, Northwestern University}
\email{aronah@math.northwestern.edu}
\address{Department of Mathematics, University of California-Irvine}
\email{wolfson@uci.edu}


\newtheorem{theorem}{Theorem}[section]
\newtheorem{theoremannounce}{Theorem}
\newtheorem{proposition}[theorem]{Proposition}
\newtheorem{corollary}[theorem]{Corollary}

\newtheorem{lemma}[theorem]{Lemma}
\theoremstyle{definition}
\newtheorem{definition}[theorem]{Definition}
\numberwithin{equation}{section}
\newtheorem{remark}[theorem]{Remark}
\newtheorem{elab}[theorem]{Elaboration}
\newtheorem{example}[theorem]{Example}
\let\pf\proof
\let\epf\endproof

\DeclareMathOperator*{\colim}{colim}

\def\limfunc#1{\mathop{\rm #1}}%

\newcommand{\elTate}{\mathsf{Tate}^{el}}
\newcommand{\elTateb}{\mathsf{Tate}^{el,\flat}}

\newcommand{\nelTate}{\mathit{n}\text{-}\mathsf{Tate}^{el}}
\newcommand{\nelTateb}{\mathit{n}\text{-}\mathsf{Tate}^{el,\flat}}

\newcommand{\Tate}{\mathsf{Tate}}
\newcommand{\Tateb}{\mathsf{Tate}^{\flat}}

\newcommand{\Tatec}{\mathsf{Tate}_{\aleph_0}}
\newcommand{\nTate}{\mathit{n}\text{-}\mathsf{Tate}}
\newcommand{\nTateb}{\mathit{n}\text{-}\mathsf{Tate}^\flat}

\newcommand{\nTatec}{\mathit{n}\text{-}\mathsf{Tate}_{\aleph_0}}

\newcommand{\Ind}{\mathsf{Ind}^a}

\newcommand{\Pro}{\mathsf{Pro}^a}
\newcommand{\Prob}{\mathsf{Pro}^{a,\flat}}

\newcommand{\lex}{\mathsf{Lex}}

\newcommand{\Gr}{Gr}
\newcommand{\Vect}{\mathsf{Vect}}

\newcommand{\Fun}{\mathsf{Fun}}
\newcommand{\into}{\hookrightarrow}
\newcommand{\onto}{\twoheadrightarrow}

\newcommand{\op}{op}

\newcommand{\Cc}{\mathcal{C}}
\newcommand{\Dc}{\mathcal{D}}
\newcommand{\Ec}{\mathcal{E}}

\newcommand{\Oc}{\mathcal{O}}

\newcommand{\Nb}{\mathbb{N}}

\newcommand{\Hom}{\underline{Hom}}

\newcommand{\rotimes}{\overrightarrow{\otimes}}

\begin{document}
\title[Tensor products and duality for Tate objects]{On the normally ordered tensor product\linebreak and duality for Tate objects}
\maketitle
\begin{abstract}
This paper generalizes the normally ordered tensor product from Tate vector spaces to Tate objects
over arbitrary exact categories. We show how to lift bi-right exact monoidal
structures, duality functors, and construct external Homs. We list some applications: (1) Ad\`{e}les of a flag
can be written as ordered tensor products; (2) Intersection numbers can be interpreted via these tensor products; (3) Pontryagin duality uniquely extends to $n$-Tate objects in locally compact abelian groups.
\end{abstract}
\section{Introduction}

Over a field $k$, the tensor product $k[[s]]\otimes k[[t]]$ is much smaller
than $k[[s,t]]$. This behaviour is usually not wanted and resolved by using a
\textsl{completed} tensor product. The corresponding problem for Laurent
series $k((s))$ and $k((t))$ is more difficult.\ It is asymmetric due to%
\begin{equation}
k((s))((t))\neq k((t))((s))\text{.}\label{lmi1}%
\end{equation}
This issue is also well-understood: Formal Laurent series can be axiomatized
as Tate vector spaces, and Beilinson and Drinfeld introduced the
\emph{normally ordered tensor product}, which has the desired property
$k((s))\overrightarrow{\otimes}k((t))=k((s))((t))$ in \cite[\S 3.6.1]%
{BeD:04}, \cite{Bei:08}.

We generalize this construction to Tate objects over arbitrary exact categories:

\begin{theoremannounce}
Let $\mathcal{C}$ be an exact category with a bi-exact monoidal structure
$\otimes$.

\begin{enumerate}
\item Then for all $n,m\geq0$, there exists a canonical bi-exact functor%
\begin{equation}
-\overrightarrow{\otimes}-:\left.  n\text{-}\mathsf{Tate}(\mathcal{C})\right.
\times\left.  m\text{-}\mathsf{Tate}(\mathcal{C})\right.  \longrightarrow
\left.  (n+m)\text{-}\mathsf{Tate}(\mathcal{C})\right.  \text{.}\tag{$%
\diamond$}\label{lmi2}%
\end{equation}
The associativity constraint for $\otimes$ naturally induces an associativity
constaint for $\overrightarrow{\otimes}$.

\item For $\mathcal{C}:=\mathsf{Vect}_{f}(k)$, the realization functor from Tate objects
in $\mathcal{C}$ to topological $k$-vector spaces sends $\overrightarrow
{\otimes}$-products to the normally ordered tensor products of \cite[\S 3.6.1]%
{BeD:04}.

\item If $\mathcal{C}$ is idempotent complete, closed monoidal and
$\underline{Hom}_{\mathcal{C}}(-,-)$ is bi-exact as well, then for all
$n,m\geq0$, there exists a canonical bi-exact functor%
\[
\underline{Hom}(-,-):\left.  n\text{-}\mathsf{Tate}(\mathcal{C})^{op}\right.
\times\left.  m\text{-}\mathsf{Tate}(\mathcal{C})\right.  \longrightarrow
\left.  (n+m)\text{-}\mathsf{Tate}(\mathcal{C})\right.  \text{.}%
\]
Moreover, there is a form of a Hom-tensor adjunction.
\end{enumerate}
\end{theoremannounce}
For precise statements, see \S \ref{sect:biexacttens} and
\S \ref{externalhoms}. We prove more, but we defer the relevant
statements to the main body of the paper. Even if $\otimes$ is symmetric,
$\overrightarrow{\otimes}$ will not be symmetric, reflecting the issue in \eqref{lmi1}.
This failure of symmetry allows one to define a different product $\rotimes_\sigma$ for every $(n,m)$-shuffle $\sigma$. We discuss this in more detail in \S\ref{sect:biexacttens}.

For general exact categories $\mathcal{C}$ new problems arise. Notably, it is rare that a monoidal structure $\otimes$ is bi-exact. To
this end, we introduce the full sub-category of `flat Tate objects', $\left.
\mathsf{Tate}^{\flat}(\mathcal{C})\right.  $, in \S \ref{sect:rightexacttens}
and show that bi-right exact monoidal structures extend to flat Tate objects
as in \eqref{lmi2}. As the category of vector spaces is split exact, none
of these problems occur in the situation of \cite[\S 3.6.1]{BeD:04}.
Next, we address duality. We show that Tate objects can also be realized inside Pro-Ind
objects, even though they are usually viewed as living inside Ind-Pro objects:
\begin{theoremannounce}
(Prop. \ref{prop:TateinProInd}) Suppose $\mathcal{C}$ is idempotent complete. Then there is a canonical fully
exact embedding%
\[
\mathsf{Tate}^{el}(\mathcal{C})\hookrightarrow\mathsf{Pro}^{a}\mathsf{Ind}%
^{a}\mathcal{C}\text{.}%
\]
The essential image of this embedding consists of all such admissible
Pro-Ind-objects which admit a lattice, i.e. a Pro sub-object with an Ind-object quotient.
\end{theoremannounce}

Moreover, duality functors always lift to Tate categories:

\begin{theoremannounce}(Prop. \ref{prop:duals})
Suppose $\mathcal{C}$ is idempotent complete. Every exact equivalence
$\mathcal{C}^{op}\overset{\sim}{\longrightarrow}\mathcal{C}$ extends to an
exact equivalence%
\[
\mathsf{Tate}^{el}(\mathcal{C})^{op}\overset{\sim}{\longrightarrow
}\mathsf{Tate}^{el}(\mathcal{C})\text{.}%
\]
This duality restricts to exact equivalences%
\[
\mathsf{Ind}^{a}(\mathcal{C})^{op}\overset{\sim}{\longrightarrow}%
\mathsf{Pro}^{a}(\mathcal{C})\qquad\text{and}\qquad\mathsf{Pro}^{a}%
(\mathcal{C})^{op}\overset{\sim}{\longrightarrow}\mathsf{Ind}^{a}%
(\mathcal{C})\text{.}%
\]
\end{theoremannounce}

These results form the foundation of a calculus of tensor products, Homs and duality
in Tate categories.\
\medskip

\begin{center} \emph{Summary of Applications}\ \end{center}
\medskip
Section \ref{sectadeles} of the paper is devoted to adelic geometry. Suppose $X/k$ is an
integral scheme of finite type. Ad\`{e}les of a flag of points $\triangle=(\eta
_{0}>\cdots>\eta_{n})$ in a scheme will be denoted by $A(\triangle
,\mathcal{O}_{X})$. The latter has a canonical structure as an object in
$\left.  n\text{-}\mathsf{Tate}(k)\right.  $. In fact, we will show that it is
a normally ordered tensor product:

\begin{theoremannounce}(Cor. \ref{cor:tubedecomp})
Suppose $X/k$ is an integral scheme of finite type and pure dimension $n$. If
$\triangle$ is a saturated flag, then%
\[
A(\triangle,\mathcal{O}_{X})=\left.  _{\eta_{n}}\Lambda_{\eta_{n-1}}\right.
\overrightarrow{\otimes}\cdots\overrightarrow{\otimes}\left.  _{\eta_{1}%
}\Lambda_{\eta_{0}}\right.  \qquad\text{for}\qquad\left.  _{x}\Lambda
_{y}\right.  :=\left.  \underset{\mathcal{F}\subseteq\mathcal{O}_{y}%
}{\operatorname*{colim}}\right.  \left.  \underset{j}{\lim}\right.
\mathcal{F}/\mathcal{I}_{x}^{j},%
\]
where $\left.  _{x}\Lambda_{y}\right.  $ are flat Tate objects in $\operatorname*{Coh}\nolimits_{\overline{\{x\}}}(X)$, $\mathcal{I}_{x}$ is
the ideal sheaf of $\overline{\{x\}}$, and $\mathcal{F}$ runs through the
coherent sub-sheaves of $\mathcal{O}_{y}$.
\end{theoremannounce}

This has the following application:\medskip

Recall that on a curve $\pi:X \rightarrow \mathbb{F}_{q}$ the action of any id\`{e}le representative of a
line bundle $L$ on the ad\`{e}les rescales any choice of a Haar measure by $q^{\operatorname{deg}(L)}$. This stays true over an arbitrary base field $k$ if one works additively and
replaces the Haar measure by a renormalized dimension theory as in Kapranov
\cite{KapranovSemiInfinite}. Phrased in terms of Tate objects: Given any
element $f\in\mathcal{O}_{y}^{\times }$, it acts through multiplication by $f$ on $\left.
_{x}\Lambda_{y}\right.  $ as a Tate object, and any such automorphism defines
a canonical $K$-theory class%
\begin{equation}\label{rf_line}
\underset{\circlearrowright\left.  _{x}\Lambda_{y}\right.  }{\left[  f\right]
}\in K_{1}(\left.  \mathsf{Tate}^{\flat}(\operatorname*{Coh}\nolimits_{\overline{\{x\}}}(X))\right.
)\text{.}%
\end{equation}
We properly introduce this notation in Definition \ref{def:UnitActingOnPartialAdeleFactor}. There is an exact push-forward functor%
\[
\pi_{\ast}:\mathsf{Tate}^{\flat}(\operatorname*{Coh}\nolimits_{0}%
(X))\longrightarrow\mathsf{Tate}(k)\text{,}%
\]
where $\operatorname*{Coh}\nolimits_{0}(X)$ are coherent sheaves of
zero-dimensional support. Using this notation, recall the following classical degree formula:
\begin{theoremannounce}\textsc{(Weil)}
Let $\pi : X \rightarrow k$ be an integral smooth proper curve with generic point $\eta_{0}$ and
$(f_{\nu\mu})_{\nu\mu}\in H^{1}(X,\mathbb{G}_{m})$ an alternating \v{C}ech representative
of a line bundle $L$ in a finite open cover $\mathfrak{U}=(U_{\alpha})_{\alpha\in I}$, $I$ totally
ordered. For any $x\in X$, let $\alpha(x)$ be the smallest element of $I$ such
that $x\in U_{\alpha(x)}$. Then%
\begin{equation*}
\deg(L)=-\sum_{\eta_{1}}\pi_{\ast} \underset{\circlearrowright\left.  _{\eta_{1}}%
\Lambda_{\eta_{0}}\right.  }{\left[  f_{\alpha(\eta_{1})\alpha(\eta_{0}%
)}\right]  }%
\end{equation*}
and the sum has only finitely many non-zero summands. The right-hand side defines an element of
$K_{1}(\left.  \mathsf{Tate}(k)\right.  )\cong\mathbb{Z}$, and this integer is
the degree of $L$. Here the sum runs over all closed points $\eta_{1}\in X$.
\end{theoremannounce}
%

We provide a higher-dimensional generalization. We define an external product in $K$-theory such
that the external product of $n$ classes $\left[f \right]$ as in \eqref{rf_line} lies in $K_{n}$ of a flat $n$-Tate category. Suppose $\pi:X \rightarrow k$ is
a purely $n$-dimensional integral smooth proper variety. Let $L_{1},\ldots,L_{n}$ be
line bundles which are represented by alternating \v{C}ech representatives%
\[
f^{q}=(f_{\rho,\nu}^{q})_{\rho,\nu\in I}\in H^{1}(X,\mathbb{G}_{m})\qquad
\quad\text{(for }q=1,\ldots,n\text{)}%
\]
in a finite open cover $\mathfrak{U}=(U_{\alpha})_{\alpha\in I}$, $I$ totally
ordered. As before, for any $x\in X$, let $\alpha(x)$ be the smallest element of $I$ such
that $x\in U_{\alpha(x)}$.

\begin{theoremannounce}(Theorem \ref{thm_multiplicity})
Under these assumptions, the sum%
\[
(-1)^{\frac{n(n+1)}{2}}\sum_{\triangle=(\eta_{0}>\cdots>\eta_{n})} \pi_{\ast }%
\underset{\circlearrowright\left.  _{\eta_{n}}\Lambda_{\eta_{n-1}}\right.
}{\left[  f_{\alpha(\eta_{n})\alpha(\eta_{n-1})}^{n}\right]  }\left.
\overrightarrow{\otimes}\right.  \cdots\left.  \overrightarrow{\otimes
}\right.  \underset{\circlearrowright\left.  _{\eta_{1}}\Lambda_{\eta_{0}%
}\right.  }{\left[  f_{\alpha(\eta_{1})\alpha(\eta_{0})}^{1}\right]  }%
\]
has only finitely many non-zero summands, defines an element in $K_{n}(\left.
n\text{-}\mathsf{Tate}(k)\right.  )\cong\mathbb{Z}$, and this integer
is the intersection multiplicity $L_{1}\cdots L_{n}$.
\end{theoremannounce}

Here the sum runs over all saturated chains $\triangle$ in $X$. For $n=1$ this is Weil's degree formula.\medskip

In Section \ref{sect:TopGroups}, we discuss an entirely different application: We develop the relation between our duality mechanism on Tate categories and Pontryagin Duality for locally compact abelian (LCA) groups, and we show how our formalism extends the Pontryagin Duality of LCA groups to a duality for $n$-Tate objects in LCA groups.

\subsection*{Acknowledgments}
We thank A. Yekutieli and M. Kapranov for helpful comments on an earlier version of this note. Moreover, we thank the anonymous referee for his or her very valuable suggestions. 

The published version of this article contained an erroneous definition of a ``shuffle product'' on a category of $\infty$-Tate objects.  We thank Dougal Davis for alerting us to our mistake.  We have removed the erroneous definition, along with the  statements which depended upon it: Corollary 2.6 and Remarks 2.7 and 2.8 in the published version.

\section{Tensor products}
Let $k$ be a field. The inclusion
\begin{equation*}
    k((t_1))\otimes_k k((t_2))\into k((t_1))((t_2))
\end{equation*}
allows us to view the 2-variable Laurent series as a completion of the tensor product of $k((t_1))\otimes_k k((t_2))$. This is an instance of a general phenomenon for Tate objects. We begin by developing the theory of bi-exact tensor products. We then extend this to consider tensor products which are only right exact.

\subsection*{Bi-exact tensor products}\label{sect:biexacttens}
\begin{proposition}\label{prop:tens1}\mbox{}
    \begin{enumerate}
        \item Let $\Cc$ be an exact category with a bi-exact monoidal structure $\otimes$. Let $n$ and $m$ be natural numbers. Then there exists a bi-exact functor
            \begin{equation*}
                -\rotimes -\colon \nTate(\Cc)\times m\text{-}\Tate(\Cc)\to (n+m)\text{-}\Tate(\Cc).
            \end{equation*}
            The associativity constraint for $\otimes$ determines an associativity constraint for $\rotimes$ in a natural fashion, and the unit object for $\otimes$ in $\Cc$ gives the unit object for $\rotimes $. Further, if $\otimes$ is faithful in both variables, then so is $\rotimes $.
        \item For $\Cc=\Vect_f(k)$, denote by
            \begin{equation*}
                \tau\colon\nTatec(k)\to\Vect_{top}(k),
            \end{equation*}
            the natural functor from countable $n$-Tate spaces to the category $\Vect_{top}(k)$ of complete, separated topological $k$-vector spaces with linear topologies. Let $V\in\nTatec(k)$ and $W\in m\text{-}\Tatec(k)$. Then there exists a canonical natural isomorphism
            \begin{equation*}
                \tau(V\rotimes  W)\cong\tau(V)\rotimes \tau(W)
            \end{equation*}
            where the tensor product on the right is the one considered by Beilinson in \cite{Bei:08} (with the same notation).\footnote{See also \cite[Section 3.6.1]{BeD:04}.}
    \end{enumerate}
\end{proposition}

We deduce the proposition from the following lemma, which we will use repeatedly throughout.
\begin{lemma}\label{lemma:induct}
    $F\colon \Cc\times\Dc\to\Ec$ be a bi-exact functor. Then $F$ extends canonically to bi-exact functors
    \begin{equation*}
        F\colon \Tate(\Cc)\times\Dc\to\Ec \qquad \text{and} \qquad F\colon \Cc\times\Tate(\Dc)\to \Tate(\Ec).
    \end{equation*}
    If $F$ is faithful in both variables, the extensions of $F$ are as well.
\end{lemma}
\pf
    By the universal property of idempotent completion, it suffices to construct bi-exact functors
    \begin{equation*}
       F\colon \elTate(\Cc)\times\Dc\to\elTate(\Ec) \qquad \text{and} \qquad F\colon \Cc\times\elTate(\Dc)\to \elTate(\Ec).
    \end{equation*}
    For the first, let $V\in\elTate(\Cc)$ and let $X\in \Dc$. Let $V\colon I\to \Pro(\Cc)$ be an elementary Tate diagram representing $V$. For each $V_i$, let $V_i\colon J_i\to \Cc$ be an admissible Pro-diagram representing $V_i$. By the bi-exactness of $F$, we have an admissible Pro-diagram
    \begin{equation*}
       F(V_i,X)\colon J_i\to\Ec
    \end{equation*}
    for each $i\in I$, and thus an elementary Tate diagram
    \begin{align*}
        F(V,X)\colon I &\to \Pro(\Ec)\\
        i &\mapsto \lim_{J_i} F(V_{i,j},X)
    \end{align*}
    as well.  Define $F(V,X)$ to be the elementary Tate object determined by this diagram.  This defines the functor $F(-,-)$ on objects.  The definition for morphisms follows from the straightening construction (cf. \cite[Lemma 3.9]{MR3510209}); further, we immediately see that if $F$ is faithful in both variables, so is its extension. To complete the proof, it remains to show that this functor is exact. But this follows by the straightening construction for exact sequences (cf. \cite[Proposition 3.12]{MR3510209}), so we are done with the first extension.

    A similar construction defines $F\colon \Cc\to\elTate(\Dc)\to\elTate(\Ec)$. In detail, let $V\in\Cc$ and let $W\in\elTate(\Dc)$. Let $W\colon I\to \Pro(\Dc))$ be an elementary Tate diagram representing $W$. For each $i\in I$, let $W_i\colon J_i\to \Dc$ be an admissible Pro-diagram representing $W_i$. Then the bi-exactness of $F$ guarantees that we have an admissible Pro-diagram
    \begin{equation*}
        F(V,W_i)\colon J_i\to\Ec,
    \end{equation*}
    and an admissible Pro-object in $\Ec$ denoted by the same. The straightening argument for exact sequences and the inductive hypothesis shows that the assignment $W_i\mapsto F(V,W_i)$ is exact, so we obtain an elementary Tate diagram
    \begin{equation*}
        F(V,W)\colon I\to \Pro(\Ec)
    \end{equation*}
    and thus an object $F(V,W)\in\elTate(\Ec)$. This constructs the functor on objects. The definition for morphisms again follows from the straightening construction for morphisms, which also shows the faithfulness in both variables. And, by the straightening argument for exact sequences, we see that it is exact in both variables as claimed.
\epf

\pf[of Proposition~\ref{prop:tens1}]
    By the universal property of idempotent completion, it suffices to construct a bi-exact functor
    \begin{equation*}
        -\rotimes -\colon \nelTate(\Cc)\times m\text{-}\elTate(\Cc)\to (n+m)\text{-}\elTate(\Cc),
    \end{equation*}
    This follows from Lemma \ref{lemma:induct} by induction on $n$ and $m$. The key point is that we first conduct all extensions in the first variable, then conduct all extensions in the second variable. The case $n=m=0$ is trivial. Now suppose that we have constructed
    \begin{equation*}
        -\rotimes -\colon (n-1)\text{-}\Tate(\Cc)\times\Cc\to(n-1)\Tate(\Cc).
    \end{equation*}
    Then by Lemma \ref{lemma:induct}, we obtain a bi-exact functor
    \begin{equation*}
        -\rotimes -\colon \nTate(\Cc)\times\Cc\to\nTate(\Cc).
    \end{equation*}
    Now suppose we have constructed
    \begin{equation*}
        -\rotimes-\colon\nTate(\Cc)\times(m-1)\text{-}\Tate(\Cc)\to(n+m-1)\text{-}\Tate(\Cc).
    \end{equation*}
    Then by Lemma \ref{lemma:induct}, we obtain a bi-exact functor
    \begin{equation*}
        -\rotimes -\colon \nTate(\Cc)\times m\text{-}\Tate(\Cc)\to(n+m)\text{-}\Tate(\Cc).
    \end{equation*}

    Denote by $\alpha$ the associativity constraint for $\otimes$. We now show that, given an $n$-Tate object $U$, an $m$-Tate object $V$, and an $\ell$-Tate object $W$, then $\alpha$ determines a natural isomorphism
    \begin{equation*}
        \alpha\colon U\rotimes (V\rotimes  W)\to^\cong(U\rotimes  V)\rotimes  W.
    \end{equation*}
    This follows from the above construction of $\rotimes $ by induction on $n$, $m$ and $\ell$. Indeed, for the base case of $n=m=\ell=0$, there is nothing to show. Further, the construction above immediately implies that an associativity constraint for $(n,0,0)$ determines an associativity constraint for $(n+1,0,0)$. Similarly, it implies that an associativity constraint for $(n,m,0)$ determines an associativity constraint for $(n,m+1,0)$. It remains to show that an associativity constraint for $(n,m,\ell)$ determines one for $(n,m,\ell+1)$. As above, let
    \begin{equation*}
        W\colon I\to \Pro((\ell-1)\text{-}\elTate(\Cc))
    \end{equation*}
    be an elementary Tate diagram representing $W$. For each $i\in I$, let
    \begin{equation*}
        W_i\colon J_i\to (\ell-1)\text{-}\elTate(\Cc)
    \end{equation*}
    be an admissible Pro-diagram representing $W_i$. By the inductive hypothesis, for each $i$, $\alpha$ determines a natural isomorphism of admissible Pro-diagrams
    \begin{equation*}
        \begin{xy}
            \Vtriangle/=`>`>/<350,500>[J_i`J_i`(n+m+\ell-1)\text{-}\elTate(\Cc);`(U\rotimes  V)\rotimes  W_i`U\rotimes  (V\rotimes  W_i)]
            \place(350,250)[\twoar(1,0)]
            \place(350,350)[_{\cong}]
        \end{xy}.
    \end{equation*}
    This determines a natural isomorphism of elementary Tate diagrams
    \begin{equation*}
        \begin{xy}
            \Vtriangle/=`>`>/<350,500>[I`I`\Pro((n+m+\ell-1)\text{-}\elTate(\Cc));`(U\rotimes  V)\rotimes  W`U\rotimes  (V\rotimes  W)]
            \place(350,250)[\twoar(1,0)]
            \place(350,350)[_{\cong}]
        \end{xy},
    \end{equation*}
    and thus a natural isomorphism of elementary Tate objects
    \begin{equation*}
        \alpha\colon (U\rotimes V)\rotimes W\to^\cong U\rotimes(V\rotimes W)
    \end{equation*}
    which completes the induction step.  A similar induction proves that the unit object for $(\Cc,\otimes)$ provides a unit object for $\rotimes$ (with the unitality constraints for $\rotimes$ determined by those for $\otimes$ under an analogous induction as the above).

    It remains to prove the comparison of topological tensor products. Let $V\in\nTatec(k)$ and $W\in m\text{-}\Tatec(k)$. We construct, by induction on $n$ and $m$, a natural isomorphism
    \begin{equation*}
        \tau(V)\rotimes\tau(W)\to^\cong \tau(V\rotimes W).
    \end{equation*}
    For $m=0$, the existence of this isomorphism follows immediately from the definition of the topology on both sides. We now show that a natural isomorphism for $(n,m)$ determines one for $(n,m+1)$. Let
    \begin{equation*}
        W\colon \Nb\to \Pro(m\text{-}\Tatec(k))
    \end{equation*}
    be an elementary Tate diagram representing $W$, and for each $i\in\Nb$, let
    \begin{equation*}
        W_i\colon \Nb\to m\text{-}\Tatec(k)
    \end{equation*}
    be an admissible Pro-diagram representing $W_i$. By definition, we have
    \begin{equation*}
        \tau(V\rotimes W)\cong\colim_{i\in\Nb}\lim_{j\in\Nb} \tau(V\rotimes W_{i,j})
    \end{equation*}
    in $\Vect_{top}(k)$. Further, by inductive hypothesis, for all $i\in\Nb$, we have a natural isomorphism of diagrams
    \begin{equation*}
        \begin{xy}
            \Vtriangle/=`>`>/<350,500>[\Nb`\Nb`\Vect_{top}(k);`\tau(V)\rotimes\tau(W_i)`\tau(V\rotimes W_i)]
            \place(350,250)[\twoar(1,0)]
            \place(350,350)[_{\cong}]
        \end{xy}.
    \end{equation*}
    This determines a natural isomorphism of diagrams
    \begin{equation*}
        \begin{xy}
            \Vtriangle/=`>`>/<350,500>[\Nb`\Nb`\Vect_{top}(k);`\lim_{j\in\Nb}\tau(V)\rotimes \tau(W_{i,j})`\tau(V\rotimes W_i)]
            \place(350,250)[\twoar(1,0)]
            \place(350,350)[_{\cong}]
        \end{xy},
    \end{equation*}
    and thus a natural isomorphism
    \begin{equation*}
        \colim_{i\in\Nb}\lim_{j\in\Nb}\tau(V)\rotimes \tau(W_{i,j})\to^\cong\tau(V\rotimes W).
    \end{equation*}
    It remains to construct a natural isomorphism
    \begin{equation*}
        \tau(V)\rotimes\tau(W)\to^\cong\colim_{i\in\Nb}\lim_{j\in\Nb}\tau(V)\rotimes\tau(W_{i,j}).
    \end{equation*}
    Denote by
    \begin{equation*}
        \delta\colon\Vect_{top}(k)\to\Vect_{top}(k)
    \end{equation*}
    the functor which sends a topological vector space to the underlying discrete vector space. Because $\otimes_k$ preserves arbitrary colimits, we have a natural map of discrete $k$-vector spaces
    \begin{equation*}
        \delta\tau(V)\otimes_k\delta\tau(W)\to\colim_{i\in\Nb}\lim_{j\in\Nb}\delta\tau(V)\otimes_k\delta\tau(W_{i,j}).
    \end{equation*}
    By definition, the tensor product $\tau(V)\rotimes\tau(W)$ is obtained by completing $\delta\tau(V)\otimes_k\delta\tau(W)$ with respect to the topology generated by subsets of the form $P\otimes_k w$ and $P\otimes Q$ for $P\subset \tau(V)$ open, $w\in W$, and $Q\subset \tau(W)$ open. By completing with respect to subsets of the form $P\otimes_k w$, we obtain
    \begin{equation*}
        \tau(V)\otimes_k\delta\tau(W).
    \end{equation*}
    By definition, an open subset $Q\subset \tau(W)$ arises as $Q\cong \colim_{i\in\Nb} Q_i$ for $Q_i\subset \tau(W_i)$ open. Therefore, completing with respect to the subsets $P\otimes_k Q$ determines a natural isomorphism
    \begin{equation*}
        \tau(V)\rotimes\tau(W)\to^\cong \colim_{i\in\Nb} \tau(V)\rotimes\tau(W_i).
    \end{equation*}
    Similarly, we see that for all $i$, we have a natural isomorphism
    \begin{equation*}
        \tau(V)\rotimes\tau(W_i)\to^\cong \lim_{j\in\Nb}\tau(V)\rotimes\tau(W_{i,j}).
    \end{equation*}
    Combining these, we obtain a natural isomorphism
    \begin{equation*}
        \tau(V)\rotimes\tau(W)\to^\cong \colim_{i\in\Nb}\lim_{j\in\Nb}\tau(V)\rotimes\tau(W_{i,j})
    \end{equation*}
    as claimed.
\epf

An $n$-Tate object is a linear algebraic analogue of an $n$-dimensional scheme $X$ equipped with a complete flag of closed subschemes
\begin{equation*}
    \xi=\left(X=Z_n\supset Z_{n-1}\cdots\supset Z_1\supset Z_0\right)
\end{equation*}
with $\dim Z_i=i$. The tensor product of Proposition \ref{prop:tens1} corresponds to the product of schemes with flags
\begin{equation*}
    ((X_1,\xi_1),(X_2,\xi_2))\mapsto(X_1\times X_2, Z_{1,n}\times Z_{2,m}\supset \cdots \supset Z_{1,n}\times Z_{2,0}\supset\cdots \supset Z_{1,0}\times Z_{2,0}).
\end{equation*}
However, we also obtain a flag in $X_1\times X_2$ for any $(n,m)$-shuffle. A similar phenomenon exists for Tate objects. To state this, we need to recall a fact about shuffles.
\begin{lemma}
    Let $\sigma$ be an $(n,m)$-shuffle, and let $\tau$ be an $(n+m,\ell)$-shuffle. Then there is a unique $(m,\ell)$-shuffle $\tau'$ and a unique $(n,m+\ell)$-shuffle $\sigma'$ such that
    \begin{equation*}
        \tau\circ(\sigma\sqcup \mathsf{id}_{\{1<\cdots<\ell\}})= \sigma'\circ(\mathsf{id}_{\{1<\cdots<n\}}\sqcup \tau').
    \end{equation*}
\end{lemma}
\pf
    Given $\sigma$ and $\tau$, we define $\tau'$ to be the $(m,\ell)$-shuffle obtained from the linearly ordered set $I$ consisting of the images, under $\tau\circ(\sigma\sqcup 1_{\{1<\cdots<\ell\}})$, of $\{1<\cdots<m\}$ and $\{1<\cdots<\ell\}$ in $\{1<\cdots<n+m+\ell\}$. Similarly, define $\sigma'$ to be the $(n,m+\ell)$-shuffle obtained by identifying the linearly ordered set $I$ with $\{1<\cdots<m+\ell\}$ and considering the images of $I$ and $\{1<\cdots<n\}$ in $\{1<\cdots<n+m+\ell\}$.
\epf

\begin{proposition}
    Let $\Cc$ be an exact category with a bi-exact symmetric monoidal structure $\otimes$. Let $n$ and $m$ be natural numbers, and let $\sigma$ be an $(n,m)$-shuffle. Then there exists a bi-exact functor
    \begin{equation*}
        -\rotimes_{\sigma}-\colon \nTate(\Cc)\times m\text{-}\Tate(\Cc)\to (n+m)\text{-}\Tate(\Cc).
    \end{equation*}
    Further, given $n$, $m$, and $\ell$, an $(n,m)$-shuffle $\sigma$ and an $(n+m,\ell)$-shuffle $\tau$, then the associativity constraint for $(\Cc,\otimes)$ determines a natural isomorphism
    \begin{equation*}
        (-\rotimes_\sigma-)\rotimes_\tau - \cong -\rotimes_{\sigma'}(-\rotimes_{\tau'}-)
    \end{equation*}
    where $\tau'$ and $\sigma'$ are as in the previous lemma.
\end{proposition}

\begin{remark}
    In the notation of the previous proposition, the functor $\rotimes$ of Proposition \ref{prop:tens1} corresponds to the trivial $(n,m)$-shuffle, i.e. the map
    \begin{equation*}
        \{1<\cdots<n\}\sqcup\{1<\cdots<m\}\to \{1<\cdots<n<1\cdots<m\}.
    \end{equation*}
\end{remark}

\pf
    We construct the functors $-\rotimes_\sigma-$ by induction on $(n,m)$. For $n,m\le 1$, the tensor products were constructed in Proposition \ref{prop:tens1}. Now suppose we have constructed $\rotimes_\sigma$ for any $(\ell,1)$-shuffle with $(\ell,1)<(n,1)$. Let $\sigma$ be an $(n,1)$-shuffle. If $\sigma$ is the trivial shuffle
    \begin{equation*}
        \sigma\colon\{1<\cdots<n\}\sqcup\{1\}\to\{1<\cdots<n<1\}
    \end{equation*}
    then the tensor product was constructed in Proposition \ref{prop:tens1}. If $\sigma(\{1\})<n$, we obtain the functor $-\rotimes_\sigma-$ by applying Lemma~\ref{lemma:induct} to the functor
    \begin{equation*}
        -\rotimes_{\sigma'}-\colon (n-1)\text{-}\Tate(\Cc)\times\Cc\to(n-1)\text{-}\Tate(\Cc)
    \end{equation*}
    where $\sigma'$ is the $(n-1,1)$-shuffle given by restricting $\sigma$ to $\{1<\cdots<n-1\}\sqcup\{1\}\subset\{1<\cdots<n\}\sqcup\{1\}$.

    Now suppose we have defined $\rotimes _\sigma$ for all $(k,\ell)$-shuffles with $(k,\ell)<(n,m)$ in reverse lexicographical ordering. Let $\sigma$ be an $(n,m)$-shuffle. Suppose $\sigma(m)=n+m$. Denote by $\sigma'$ the $(n,m-1)$-shuffle given by restricting $\sigma$ to $\{1<\cdots<n\}\sqcup\{1<\cdots<m-1\}$. We now obtain the functor $-\rotimes_\sigma-$ by applying Lemma~\ref{lemma:induct} to the functor
    \begin{equation*}
        -\rotimes_{\sigma'}-\colon\nTate(\Cc)\times(m-1)\text{-}\Tate(\Cc)\to(n+m-1)\text{-}\Tate(\Cc).
    \end{equation*}
    If $n+m\notin\sigma(\{1<\cdots<m\})$, then we define $\rotimes _{\sigma}$ by applying Lemma~\ref{lemma:induct} to the functor
    \begin{equation*}
        -\rotimes_{\sigma'}-\colon(n-1)\text{-}\Tate(\Cc)\times m\text{-}\Tate(\Cc)\to(n+m-1)\text{-}\Tate(\Cc)
    \end{equation*}
    where now $\sigma'$ is the $(n-1,m)$-shuffle obtained by restricting $\sigma$ to $\{1<\cdots<n-1\}\sqcup\{1<\cdots<m\}$.

    It remains to establish the associativity. This follows by inductive argument analogous to the one which showed the associativity of the tensor product of Proposition \ref{prop:tens1}.
\epf

\subsection*{Right exact tensor products}\label{sect:rightexacttens}
We now consider the situation when $\Cc$ has only a right exact tensor product.  We show that we still obtain a good category of ``flat'' admissible Pro-objects, and, from this, good categories of ``flat'' $n$-Tate objects on which we can define bi-exact tensor products. A primary motivation for this is to be able to work with adically complete modules, which are not, in general, Pro-objects in a category of flat modules.

\begin{definition}
    Let $\Cc$ be an exact category.  We say that a tensor product $\otimes$ is \emph{bi-right exact} if, for any $A\in \Cc$ and for any short exact sequence
    \begin{equation*}
        0\to X\into^i Y\onto^p Z\to 0,
    \end{equation*}
    the maps
    \begin{align*}
        A\otimes Y &\to^{1\otimes p} A\otimes Z\\
        Y\otimes A &\to^{p\otimes 1} Z\otimes A
    \end{align*}
    are admissible epics, and the natural maps
    \begin{align*}
        A\otimes X &\to \ker(1\otimes p)\\
        X\otimes A &\to \ker(p\otimes 1)
    \end{align*}
    are admissible epics as well.
\end{definition}

A bi-right exact tensor product $\otimes$ on $\Cc$ canonically extends to a bi-right exact tensor product
\begin{equation*}
    \hat{\otimes}\colon\Pro(\Cc)\times\Pro(\Cc)\to\Pro(\Cc).
\end{equation*}

\begin{definition}\label{def:flatadmpro}
    Define the category $\Prob(\Cc)$ of \emph{flat admissible Pro-objects} to be the full sub-category of $\Pro(\Cc)$ consisting of all $X\in\Pro(\Cc)$ such that the functors
    \begin{align*}
        -\hat{\otimes} X\colon\Pro(\Cc)&\to\Pro(\Cc)\\
        X\hat{\otimes} -\colon\Pro(\Cc)&\to\Pro(\Cc)
    \end{align*}
    are exact.
\end{definition}

\begin{proposition}\label{prop:prob}
    Let $\Cc$ be an exact category with a bi-right exact tensor product $\otimes$. Then:
    \begin{enumerate}
        \item The sub-category $\Prob(\Cc)\subset\Pro(\Cc)$ is closed under extensions.  In particular, $\Prob(\Cc)$ is a fully exact sub-category of $\Pro(\Cc)$.
        \item The functors
            \begin{align*}
                -\hat{\otimes}-\colon\Pro(\Cc)\times\Prob(\Cc)&\to\Pro(\Cc)\\
                -\hat{\otimes}-\colon\Prob(\Cc)\times\Pro(\Cc)&\to\Pro(\Cc)\intertext{are bi-exact, and restrict to a bi-exact functor}
                -\hat{\otimes}-\colon\Prob(\Cc)\times\Prob(\Cc)&\to\Prob(\Cc).
            \end{align*}
    \end{enumerate}
\end{proposition}
\pf
    We first show that $\Prob(\Cc)$ is closed under extensions. Let
    \begin{equation}\label{flatclosed1}
        0\to X\to Y\to Z\to 0
    \end{equation}
    be a short exact sequence of admissible Pro-objects with $X$ and $Z$ in $\Prob(\Cc)$. Let
    \begin{equation}\label{flatclosed2}
        0\to A\to B\to C\to 0
    \end{equation}
    be any short exact sequence in $\Pro(\Cc)$. The category $\lex(\Cc)$ of left exact functors to abelian groups is a Grothendieck abelian category, see \cite[Definition 2.20]{MR3510209}. In particular, it has enough injectives. By duality, $\lex(\Cc^{\op})^{\op}$ has enough projectives. By \cite[Theorem 4.2]{MR3510209}, $\Pro(\Cc)$ is a fully exact sub-category of $\lex(\Cc^{\op})^{\op}$. We can therefore choose a projective resolution in $\lex(\Cc^{\op})^{\op}$
    \begin{equation*}
        0\to A_\bullet\to B_\bullet\to C_\bullet\to 0
    \end{equation*}
    of the sequence \eqref{flatclosed2}. Tensoring with \eqref{flatclosed1}, we get a $3\times 3$ diagram with exact rows and columns
    \begin{equation*}
        \xymatrix{
        & 0 \ar[d] & 0 \ar[d] & 0\ar[d]  \\
        0 \ar[r] & A_\bullet\hat{\otimes} X \ar[r] \ar[d] &  A_\bullet\hat{\otimes} Y \ar[r] \ar[d] & A_\bullet\hat{\otimes} Z \ar[r] \ar[d] & 0\\
        0 \ar[r] & B_\bullet\hat{\otimes} X \ar[r] \ar[d] &  B_\bullet\hat{\otimes} Y \ar[r] \ar[d] & B_\bullet\hat{\otimes} Z \ar[r] \ar[d] & 0\\
        0 \ar[r] & C_\bullet\hat{\otimes} X \ar[r] \ar[d] &  C_\bullet\hat{\otimes} Y \ar[r] \ar[d] & C_\bullet\hat{\otimes} Z \ar[r] \ar[d] & 0\\
        & 0 & 0 & 0
        }
    \end{equation*}
    The left and right columns consist solely of acyclic complexes because $X$ and $Z$ are flat admissible Pro-objects. Thus, by the exactness of the above diagram, the middle column consists solely of acyclic complexes.  The same argument applies to tensoring on the left. We conclude that $Y\in\Prob(\Cc)$.

    It remains to show the bi-exactness.  By the definition of $\Prob(\Cc)$, the functor
    \begin{equation}\label{flatex1}
        -\hat{\otimes}-\colon\Pro(\Cc)\times\Prob(\Cc)\to\Pro(\Cc)
    \end{equation}
    is exact in the first variable. It remains to show that it is exact in the second variable. To see this, observe that the functor $-\hat{\otimes}-$ is (by construction) the restriction of the right Kan extension of $\otimes$ to $\lex(\Cc^{\op})^{\op}$. Because $\lex(\Cc^{\op})^{\op}$ has enough projectives, to compute $-\hat{\otimes}^{\mathbb{L}}-$, we can resolve in either variable.  In particular, given $A\in\Pro(\Cc)$ and a short exact sequence in $\Prob(\Cc)$
    \begin{equation*}
        0\to X\to Y\to Z\to 0
    \end{equation*}
    we obtain a short exact sequence of complexes
    \begin{equation*}
        0\to A_\bullet\hat{\otimes} X\to A_\bullet\hat{\otimes} Y\to A_\bullet\hat{\otimes}  Z\to 0
    \end{equation*}
    for any projective resolution $A_\bullet$ of $A$. However, because the functors $-\hat{\otimes}X$, $-\hat{\otimes}Y$, and  $-\hat{\otimes}Z$ are exact, each of the complexes in the above sequence is acyclic. In particular, they are all equivalent to their $H_0$, and we have an exact sequence
    \begin{equation*}
        0\to A\hat{\otimes} X\to A\hat{\otimes} Y\to A\hat{\otimes}  Z\to 0.
    \end{equation*}
    We conclude that the functor \eqref{flatex1} is bi-exact, and the same argument shows the bi-exactness of
    \begin{equation*}
        -\hat{\otimes}-\colon\Prob(\Cc)\times\Pro(\Cc)\to\Pro(\Cc).
    \end{equation*}
    Finally, that $A\hat{\otimes}B\in\Prob(\Cc)$ if $A,B\in\Prob(\Cc)$ follows directly from the associativity of $\hat{\otimes}$.
\epf

\begin{definition}
    Let $\Cc$ be an exact category with a bi-right exact tensor product. Define the category $\elTateb(\Cc)$ of \emph{flat elementary Tate objects} to be the full sub-category of $\elTate(\Cc)$ consisting of all elementary Tate objects which can be represented by an admissible Ind-diagram in $\Prob(\Cc)$.
\end{definition}

\begin{proposition}
    The sub-category $\elTateb(\Cc)\subset\elTate(\Cc)$ is closed under extensions.  In particular, $\elTateb(\Cc)$ is a fully exact sub-category of $\elTate(\Cc)$.
\end{proposition}
\pf
    Let
    \begin{equation*}
        0\to X\to Y\to Z\to 0
    \end{equation*}
    be a short exact sequence of elementary Tate objects with $X,Z\in\elTateb(\Cc)$. By straightening exact sequences \cite[Proposition 3.12]{MR3510209}, we can represent this as the colimit of an admissible Ind-diagram of short exact sequences in $\Pro(\Cc)$
    \begin{equation*}
        0\to X_i\to Y_i\to Z_i\to 0
    \end{equation*}
    with $X_i$ and $Z_i$ in $\Prob(\Cc)$. Because $\Cc\subset\Pro(\Cc)$ is closed under extensions (\cite[Theorem 4.2(2)]{MR3510209}), the $Y_i$ are lattices of $Y$. Further, because $\Prob(\Cc)\subset\Pro(\Cc)$ is closed under extensions (Proposition \ref{prop:prob}), the $Y_i$ are flat admissible Pro-objects, i.e. $Y\in\elTateb(\Cc)$.
\epf

\begin{definition}
    Define the category $\Tateb(\Cc)$ of \emph{flat Tate objects} to be the idempotent completion of $\elTateb(\Cc)$. For $n>1$, define the category $\nelTateb(\Cc)$ of \emph{flat elementary $n$-Tate objects} to be the category of elementary $(n-1)$-Tate objects in $\Tateb(\Cc)$. Define the category $\nTateb(\Cc)$ of \emph{flat $n$-Tate objects} to be the idempotent completion of $\nelTateb(\Cc)$ (equivalently, the category of $(n-1)$-Tate objects in $\Tateb(\Cc)$).
\end{definition}

\begin{proposition}
    Let $\Cc$ be an exact category with a bi-right exact tensor product $\otimes$. Then for any $n,m$, there exists a bi-exact functor
    \begin{align*}
        -\rotimes-\colon\nTateb(\Cc)\times m\text{-}\Tate(\Cc)&\to(n+m)\text{-}\Tate(\Cc)\\
        \intertext{which restricts to a bi-exact functor}
        -\rotimes -\colon \nTateb(\Cc)\times m\text{-}\Tateb(\Cc)&\to (n+m)\text{-}\Tateb(\Cc).
    \end{align*}
\end{proposition}

\pf
    We prove these statements by an induction on $n$ and $m$ which follows the same logic, {\em mutatis mutandis}, as the proofs of Lemma~\ref{lemma:induct} and Proposition~\ref{prop:tens1}.  We leave the details to the interested reader.
\epf

\section{Duality}
Let $V$ be an elementary Tate object in $\Cc$, and consider the diagram
\begin{align*}
    \Gr(V)&\to\Ind(\Cc)\\
    L&\mapsto V/L
\end{align*}
When $\Cc$ is idempotent complete, the poset $\Gr(V)$ is co-filtered by \cite[Theorem 6.7]{MR3510209}, and the above defines an admissible Pro-diagram in $\Ind(\Cc)$. In this section, we will show that the assignment
\begin{equation*}
    V\mapsto \lim_{\Gr(V)} V/L
\end{equation*}
defines a fully exact embedding
\begin{equation*}
    \Phi\colon\elTate(\Cc)\to \Pro(\Ind(\Cc)).
\end{equation*}
This embedding sends Pro-objects to Pro-objects and Ind-objects to Ind-objects. From this, we will deduce that if $\Cc$ is idempotent complete, then $\Ind(\Cc)$ is right s-filtering in $\elTate(\Cc)$ (thus clearing up a loose end in \cite{MR3510209}). More important, we deduce that if $\Cc$ has an exact duality, then so does $\elTate(\Cc)$.

\begin{proposition}\label{prop:TateinProInd}
    Let $\Cc$ be idempotent complete. Then the assignment above defines a fully exact embedding
    \begin{equation}\label{TateinProInd}
        \Phi\colon\elTate(\Cc)\into\Pro(\Ind(\Cc)).
    \end{equation}
    The essential image of this embedding consists of all admissible Pro-Ind objects which admit a lattice, or equivalently, by \cite[Theorem 5.6]{MR3510209}, the essential image is the sub-category
    \begin{equation*}
        \elTate(\Cc^{\op})^{\op}\subset\Ind(\Pro(\Cc^{\op}))^{\op}=\Pro(\Ind(\Cc)).
    \end{equation*}
    Further, this embedding restricts to the canonical embeddings
    \begin{align*}
        \Ind(\Cc)&\into\Pro(\Ind(\Cc))\intertext{and}
        \Pro(\Cc)&\into\Pro(\Ind(\Cc)).
    \end{align*}
\end{proposition}
\pf
    We start by showing that the assignment above is functorial. Indeed, let
    \begin{equation*}
        f\colon V\to W
    \end{equation*}
    be a map of elementary Tate objects. Let $\Delta^1$ denote the category with two objects $0$ and $1$ and one non-identity morphism $0\to 1$. Recall that the inclusion
    \begin{equation*}
        \elTate(\Fun(\Delta^1,\Cc))\into\Fun(\Delta^1,\elTate(\Cc))
    \end{equation*}
    is an exact equivalence (by the straightening construction). Under this equivalence, we can view the map $f$ as an elementary Tate object in $\Fun(\Delta^1,\Cc)$. Note that a lattice in $V\to^f W$ consists of a lattice $L$ of $V$ and a lattice $L'$ of $W$ fitting into a commuting square
    \begin{equation*}
        \begin{xy}
            \square[L`L'`V`W;```f].
        \end{xy}
    \end{equation*}
    The idempotent completeness of $\Cc$ implies that $\Fun(\Delta^1,\Cc)$ is also idempotent complete. Therefore, the Sato Grassmannian $\Gr(f\colon V\rightarrow W)$ is also co-directed, and the assignment above determines a map of Pro-Ind objects
    \begin{equation}\label{Gr(f)1}
        \lim_{\Gr(f\colon V\rightarrow W)} \left(V/L\to W/L'\right).
    \end{equation}
    The projection $\Gr(f\colon V\rightarrow W)\to \Gr(V)$ determines a map
    \begin{equation*}
        \lim_{\Gr(f\colon V\rightarrow W)} V/L \to \lim_{\Gr(V)} V/L.
    \end{equation*}
    It remains to show that this map is an isomorphism, or equivalently that the projection $\Gr(f\colon V\rightarrow W)\to \Gr(V)$ is cofinal. However, this is obvious, as the projection is surjective, since every lattice of $V$ factors through some lattice of $W$ (because Pro-objects are left filtering in Tate objects). We can therefore define the functor $\Phi$ on morphisms by specifying that $f$ is sent to the map
    \begin{equation*}
        \lim_{\Gr(V)}V/L\to^\cong\lim_{\Gr(f\colon V\rightarrow W)} V/L\to\lim_{\Gr(W)} W/L'
    \end{equation*}
    where the last map is the composite of \eqref{Gr(f)1} with the map of limits induced from the projection $\Gr(f\colon V\rightarrow W)\to \Gr(W)$. The associativity of this assignment now follows by similar lines as above, and unitality is immediate; i.e. we have indeed defined a functor.

    We now show that the functor we have constructed is a fully exact embedding, i.e. an exact equivalence onto its essential image.  For this, we observe that, after passing to opposite categories, \cite[Theorem 5.6]{MR3510209} shows that the category \begin{align*}
        \elTate(\Cc^{\op})^{\op}&\subseteq \Ind(\Pro(\Cc^{\op}))^{\op} \\
        &=:\Pro(\Pro(\Cc^{\op})^{\op})\\
        &=:\Pro(\Ind(\Cc))
    \end{align*}
    is the smallest full sub-category of $\Pro(\Ind(\Cc))$ which contains $\Pro(\Cc)$, $\Ind(\Cc)$ and is closed under extensions. In particular, the construction of the functor \eqref{TateinProInd} shows that it factors through the sub-category $\elTate(\Cc^{\op})^{\op}$. Because $\Cc$ is idempotent complete if and only if $\Cc^{\op}$ is, we have also defined a functor
    \begin{align*}
        \Phi^{\op}\colon \elTate(\Cc^{\op})^{\op}\to & \Pro(\Ind(\Cc^{\op}))^{\op}\\
        &=:\Ind(\Ind(\Cc^{\op})^{\op})\\
        &=:\Ind(\Pro(\Cc)).
    \end{align*}
    Unwinding the definitions, we see that $\Phi^{\op}$ and $\Phi$ are inverse equivalences.

    It remains to show that $\Phi$ preserves exact sequences. To see this, recall the equivalence $\Ec\elTate(\Cc)\simeq\elTate(\Ec\Cc)$ \cite[Proposition 5.12]{MR3510209}, and observe that $\Ec\Cc$ is idempotent complete if and only if $\Cc$ is. The construction above therefore defines an equivalence
    \begin{equation*}
        \Ec\elTate(\Cc)\simeq\elTate(\Ec\Cc)\to^{\Phi}_\simeq \elTate(\Ec\Cc^{\op})^{\op} \simeq (\Ec\elTate(\Cc^{\op}))^{\op}\subset \Ec(\Pro(\Ind(\Cc)),
    \end{equation*}
    i.e. $\Phi$ is fully exact.
\epf

\begin{corollary}
    Let $\Cc$ be idempotent complete, and let $V\in\elTate(\Cc)$. Then the intersection $\lim_{\Gr(V)} L$ of all lattices in $V$  is the zero object.
\end{corollary}
\pf
    A cone on $\Gr(V)$ is equivalent to a map $f\colon X\to V$ such that $f$ factors through $L\into V$ for every lattice $L\in\Gr(V)$. The definition of the functor \eqref{TateinProInd} shows that $\Phi(f)$ is the zero map. But, $\Phi$ is fully faithful, so therefore $f$ is the zero map and $0\to V$ is the terminal cone on $\Gr(V)$.
\epf

\begin{remark}
    We interpret the corollary as follows.  Tate objects are a categorical abstraction of locally linearly compact topological vector spaces.  In such a vector space, lattices form a basis of neighborhoods of the origin; in particular the intersection of all lattices is the origin itself.  The corollary shows that the same holds for lattices in arbitrary Tate objects, provided the category $\Cc$ is idempotent complete.
\end{remark}

\begin{corollary}
    If $\Cc$ is idempotent complete, then $\Ind(\Cc)$ is right s-filtering in $\elTate(\Cc)$.
\end{corollary}
\pf
    The sub-category $\Ind(\Cc)$ is right filtering in $\elTate(\Cc)$ by \cite[Proposition 5.10 (2)]{MR3510209}. Further, because $\Ind(\Cc)$ is right special in $\Pro(\Ind(\Cc))$ it is right special in $\elTate(\Cc)$.
\epf

\begin{proposition}\label{prop:duals}
    Let $\Cc$ be idempotent complete. An exact equivalence $\Cc^{\op}\to^\simeq \Cc$ extends to an exact equivalence
    \begin{equation*}
        \elTate(\Cc)^{\op}\to^\simeq \elTate(\Cc).
    \end{equation*}
    This duality restricts to exact equivalences
    \begin{align*}
        \Ind(\Cc)^{\op}&\to^\simeq\Pro(\Cc)\intertext{and}
        \Pro(\Cc)^{\op}&\to^\simeq\Ind(\Cc).
    \end{align*}
If $\mathcal{D},\mathcal{D}^{\prime}$ are full sub-categories such that the
equivalence restricts to $\mathcal{D}^{op}\rightarrow\mathcal{D}^{\prime}$ and
$\mathcal{D}^{\prime op}\rightarrow\mathcal{D}$, then this property is preserved:
\begin{equation*}
        \elTate(\mathcal{D})^{\op}\to^\simeq \elTate(\mathcal{D^{\prime }})\qquad
        \elTate(\mathcal{D^{\prime }})^{\op}\to^\simeq \elTate(\mathcal{D}).
    \end{equation*}
\end{proposition}
\pf
    The duality is given by the composition of exact equivalences
    \begin{equation*}
        \elTate(\Cc)^{\op}\to^\Phi_\simeq \elTate(\Cc^{\op})\to^\simeq \elTate(\Cc)
    \end{equation*}
    where the first equivalence is that induced by $\Phi$, and the second is that induced by the duality on $\Cc$ (cf. \cite[Proposition 5.16]{MR3510209}). The statement about restrictions is now immediate.
\epf

\begin{corollary}
\label{cor_induce_selfequivalences}Suppose $\mathcal{C}$ is an exact category
with an exact equivalence $\Cc^{\op}\to^\simeq \Cc$. For every $n\geq0$, there is a canonical exact equivalence%
\[
\left.  n\text{-}\mathsf{Tate}(\mathcal{C}^{ic})^{op}\right.  \overset{\sim
}{\longrightarrow}\left.  n\text{-}\mathsf{Tate}(\mathcal{C}^{ic})\right.  \text{.}%
\]

\end{corollary}

\pf
    This follows by induction. Firstly, if $(-)^{\vee}:\mathcal{C}^{op}\overset{\sim}{\rightarrow}\mathcal{C}$ is an exact equivalence, then the idempotent completion $\mathcal{C}^{ic}$ also carries such an exact equivalence $(\mathcal{C}^{ic})^{op}\overset{\sim}{\rightarrow}\mathcal{C}^{ic}$, by sending $(X,p)$ to $(X^{\vee},p^{\vee})$, where $p^{\vee}$ is the idempotent induced on the dual \cite[\S 5.1]{MR2600285}. By Proposition \ref{prop:duals} for every idempotent complete exact category, we get an equivalence on elementary Tate objects, $\elTate(\mathcal{C})^{op}\overset{\sim}{\longrightarrow}\elTate(\mathcal{C})$. Taking $\mathsf{Tate}^{el}(\mathcal{C})$ for $\mathcal{C}$ and repeating the above $n$ times yields the claim.
\epf

\section{External Homs}\label{externalhoms}

We now consider extensions of internal homs in $\Cc$ to higher Tate objects. If $U\in\nTate(\Cc)$ and $V\in m\text{-}\Tate(\Cc)$, we will construct an $n+m$-Tate object $\underline{Hom}(V,W)$,  which we think of as an ``external hom''.\footnote{These homs are `external' in a sense analogous to the external tensor product of modules over different $k$-algebras.}  We then explain in what sense the external hom $\Hom(U,-)$ provides a right adjoint to $U\rotimes -$ when $U$ is a (higher) Tate object. If $\Cc$ is in fact a rigid tensor category, we also show that, in a natural sense, $-\rotimes  U^\vee$ is right adjoint to $-\rotimes  U$. 

\begin{proposition}
    Let $\Cc$ be an idempotent complete, closed monoidal, exact category in which both $-\otimes-$ and $\Hom_{\Cc}(-,-)$ are bi-exact. Then there exists a bi-exact functor
    \begin{equation*}
        \Hom(-,-)\colon \nTate(\Cc)^{\op}\times m\text{-}\Tate(\Cc)\to (n+m)\text{-}\Tate(\Cc).
    \end{equation*}
\end{proposition}

\begin{remark}
    The construction of the functors $\Hom$ follows by an induction similar to the proof of Proposition \ref{prop:tens1}. The key is that the diagrams defining the Tate objects in the first variable sandwich the diagrams defining the Tate objects in the second variable, i.e. 
    \begin{equation*}
        \hom_{\Tate(\Cc)}(\colim_I\lim_{J_i} V_{ij},\colim_K\lim_{L_k}W_{kl})=\lim_I\colim_K\lim_{L_k}\colim_{J_i}\hom_{\Cc}(V_{ij},W_{kl}).
    \end{equation*}
    In particular, while the logic is very similar, we cannot just quote Lemma \ref{lemma:induct}.
\end{remark}
\pf
    By the universal property of idempotent completeness, it suffices to construct $\Hom$ for elementary Tate objects. We construct the functors using three related inductions on $(n,m)$.  First, we show that the construction for $(n,n)$ implies it for $(n+m,n)$ for all $m$. Second, we show that the construction for $(n,n)$ implies it for $(n,n+m)$ for all $m$. And last, we show that the construction for $(n-1,n-1)$ implies it for $(n,n)$.

    For the base case $n=m=0$, there is nothing to show. Now suppose that we have constructed $\Hom(-,-)$ for $(n+m,n)$. Let $W\in\nelTate(\Cc)$ and let $V\in(n+m+1)\text{-}\elTate(\Cc)$ be represented by an elementary Tate diagram
    \begin{equation*}
        V\colon I\to\Pro((n+m)\text{-}\elTate(\Cc)).
    \end{equation*}
    For $i\in I$, let $V_i$ be represented by an admissible Pro-diagram
    \begin{equation*}
        V_i\colon J_i\to(n+m)\text{-}\elTate(\Cc).
    \end{equation*}
    Then we define
    \begin{equation*}
        \Hom(V,W):=\lim_{I^{\op}}\colim_{J_i^{\op}}\Hom(V_{i,j},W)\in\Pro(\Ind((2n+m)\text{-}\elTate(\Cc))).
    \end{equation*}
    This is invariant under different choices of representing diagrams. By the straightening construction for morphisms, we see the assignment $(V,W)\mapsto\Hom(V,W)$ is indeed functorial. And, by the inductive hypothesis and the straightening construction for exact sequences, we see that it is bi-exact. In order to show that it factors through the inclusion
    \begin{equation*}
        (2n+m+1)\text{-}\elTate(\Cc)\into\Pro(\Ind((2n+m)\text{-}\elTate(\Cc))),
    \end{equation*}
    it suffices to exhibit a co-lattice of $\Hom(V,W)$. And indeed, the construction shows that
    \begin{equation*}
        \Hom(V,W)\onto\Hom(L,W)
    \end{equation*}
    is a co-lattice, for any lattice $L\into V$. This completes the first induction.

    The second induction follows by a similar argument. Suppose we have constructed $\Hom(-,-)$ for $(n,n+m)$. Let $V\in\nelTate(\Cc)$ and let $W\in(n+m+1)\text{-}\elTate(\Cc)$ be represented by an elementary Tate diagram
    \begin{equation*}
        W\colon I\to\Pro((n+m)\text{-}\elTate(\Cc)).
    \end{equation*}
    For $i\in I$, let $V_i$ be represented by an admissible Pro-diagram
    \begin{equation*}
        W_i\colon J_i\to(n+m)\text{-}\elTate(\Cc).
    \end{equation*}
    Then we define
    \begin{equation*}
        \Hom(V,W):=\colim_I\lim_{J_i}\Hom(V,W_{i,j})\in\Ind(\Pro((2n+m)\text{-}\elTate(\Cc))).
    \end{equation*}
    As above, this assignment is functorial and bi-exact. We see that it takes values in the category $2n+m+1$-Tate objects by observing that for any lattice $L\into W$, the inclusion
    \begin{equation*}
        \Hom(V,L)\into\Hom(V,W)
    \end{equation*}
    is a lattice.

    It remains to show the third induction. Suppose that we have constructed $\Hom(-,-)$ for $(n-1,n-1)$. Let $V, W\in\nelTate(\Cc)$, and let
    \begin{align*}
        V\colon I \to\Pro((n-1)\text{-}\elTate(\Cc)) \qquad \text{and} \qquad
        W\colon L \to\Pro((n-1)\text{-}\elTate(\Cc))
    \end{align*}
    be elementary Tate diagrams representing $V$ and $W$. For each $i\in I$ and $\ell\in L$, let
    \begin{align*}
        V_i\colon J_i &\to\Pro((n-1)\text{-}\elTate(\Cc))\intertext{and}
        W_\ell\colon K_\ell &\to\Pro((n-1)\text{-}\elTate(\Cc))
    \end{align*}
    be admissible Pro-diagrams representing $V_i$ and $W_\ell$. Then we define
    \begin{align*}
        \Hom(V,W)&:=\lim_{I^{\op}}\colim_L\lim_{K_\ell}\colim_{J_i^{\op}}\Hom(V_{i,j},W_{\ell,k})\\
        &\in \Pro(\Ind(\Pro(\Ind((2n-2)\text{-}\elTate(\Cc))))).
    \end{align*}
    This assignment is invariant under different choices of representing diagrams. By the straightening construction for morphisms, we see the assignment $V\mapsto\Hom(V,W)$ is indeed functorial. And, by the inductive hypothesis and the straightening construction for exact sequences, we see that it is bi-exact.

    It remains to show that it factors through the sub-category of $2n$-Tate objects. For any lattice $L\into V$, we obtain, by the bi-exactness of $\Hom(-,-)$, an exact sequence
    \begin{equation*}
        \Hom(V/L,W)\into \Hom(V,W)\onto\Hom(L,W).
    \end{equation*}
    Because the sub-category of elementary $2n$-Tate objects is closed under extensions in
\[
\Pro(\Ind(\Pro(\Ind((2n-2)\text{-}\elTate(\Cc)))))
\]
by \cite[Theorem 5.6]{MR3510209} and Proposition \ref{prop:TateinProInd}, it suffices to show that the kernel and cokernel terms of this exact sequence are $2n$-Tate objects. We see this directly, as follows. Let
    \begin{equation*}
        V/L\colon I\to (n-1)\text{-}\Tate(\Cc)
    \end{equation*}
    be an admissible Ind-diagram representing $V/L$. Then, by definition, we have
    \begin{align*}
        \Hom(V/L,W)\cong &\lim_{I^{\op}}\Hom((V/L)_i,W)\\
        &\in\Pro((2n-1)\text{-}\elTate(\Cc))\\
        &\subset(2n)\text{-}\elTate(\Cc)
    \end{align*}
    where the assertion that this is a Pro-object follows by inductive hypothesis and the first statement we showed. Similarly, let
    \begin{align*}
        L\colon J &\to(n-1)\text{-}\Tate(\Cc)\intertext{be an admissible Pro-diagram representing $L$. Let}
        W\colon A &\to \Pro((n-1)\text{-}\Tate(\Cc))\intertext{be an elementary Tate diagram representing $W$, and for each $a\in A$ let}
        W_a\colon B_a&\to (n-1)\text{-}\Tate(\Cc)
    \end{align*}
    be an admissible Pro-diagram representing $W_a$. Then, for each $a\in A$, we have
    \begin{align*}
        \Hom(L,W_a)\cong &\lim_{B_a}\colim_{J^{\op}}\Hom(L_j,W_{a,b})\\
        &\in \Pro(\Ind((2n-2)\text{-}\Tate(\Cc)))\\
        &\subset \Pro((2n-1)\text{-}\Tate(\Cc)).
    \end{align*}
    Thus, the assignment $a\mapsto \Hom(L,W_a)$ defines an admissible Ind-diagram in $\Pro((2n-1)\text{-}\Tate(\Cc))$. To see that it is in fact an elementary Tate diagram, we observe that, for each $a<a'$, we have an exact sequence
    \begin{equation*}
        \Hom(L,W_a)\into\Hom(L,W_{a'})\onto\Hom(L,W_a/W_{a'}),
    \end{equation*}
    Because $W_a/W_{a'}\in(n-1)\text{-}\Tate(\Cc)$ and $L\in\Pro((n-1)\text{-}\Tate(\Cc))$, we conclude, by inductive hypothesis, that the cokernel $\Hom(L,W_a/W_{a'})$ is a $(2n-1)$-Tate object, and therefore that
    \begin{equation*}
        \Hom(L,W)\cong\colim_A\Hom(L,W_a)
    \end{equation*}
    is an elementary $2n$-Tate object as claimed.
\epf

\begin{proposition}\label{prop:homclosed}
    Let $\Cc$ be an idempotent complete, closed monoidal exact category in which both $-\otimes-$ and $\Hom_{\Cc}(-,-)$ are bi-exact.
    \begin{enumerate}
        \item Given $U\in\Cc$, and $V,W\in\nTate(\Cc)$, there exists a canonical natural isomorphism
            \begin{equation*}
                \hom_{\nTate(\Cc)}(U\rotimes  V,W)\cong\hom_{\nTate(\Cc)}(V,\Hom(U,W)).
            \end{equation*}
        \item Given $V\in\Cc$, and $U,W\in\nTate(\Cc)$, there exists a canonical natural isomorphism
            \begin{equation*}
                \hom_{\nTate(\Cc)}(U\rotimes  V,W)\cong\hom_{2n\text{-}\Tate(\Cc)}(V,\Hom(U,W)).
            \end{equation*}
    \end{enumerate}
\end{proposition}
\pf
    The proof follows by induction on $n$ and formal manipulation of limits and colimits. For $n=0$, there is nothing to show.  Now suppose we have shown the statement for $n$. In the first case, suppose $U\in\Cc$ and let $V,W\in(n+1)\text{-}\Tate(\Cc)$ be represented by elementary Tate diagrams
    \begin{align*}
        V\colon I \to \Pro(\nTate(\Cc)) \qquad \text{and} \qquad W\colon A \to \Pro(\nTate(\Cc))
    \end{align*}
    and for each $i\in I$ and $a\in A$, let $V_i$ be represented by an admissible Pro-diagram
    \begin{align*}
        V_i\colon J_i &\to \nTate(\Cc)\intertext{and let $W_a$ be represented by an admissible Pro-diagram}
        W_a\colon B_a &\to \nTate(\Cc).
    \end{align*}
    Then we have
    \begin{align*}
        \hom_{(n+1)\text{-}\Tate(\Cc))(\Cc)}(U\rotimes  V,W)&\cong \lim_{I^{\op}}\colim_A\lim_{B_a}\colim_{J_i^{\op}}\hom_{\nTate(\Cc)}(U\rotimes  V_{i,j},W_{a,b})\\
        &\cong \lim_{I^{\op}}\colim_A\lim_{B_a}\colim_{J_i^{\op}}\hom_{\nTate(\Cc)}(V_{i,j},\Hom(U,W_{a,b}))\\
        &\cong \hom_{(n+1)\text{-}\Tate(\Cc))}(V,\Hom(U,W)).
    \end{align*}
    Similarly, if $V\in\Cc$ and $U,W\in (n+1)\text{-}\Tate(\Cc)$ are represented by elementary Tate diagrams
    \begin{align*}
        U\colon I \to \Pro(\nTate(\Cc))\qquad \text{and}\qquad  W\colon A \to \Pro(\nTate(\Cc))
    \end{align*}
    and for each $i\in I$ and $a\in A$, let $U_i$ be represented by an admissible Pro-diagram
    \begin{align*}
        U_i\colon J_i &\to \nTate(\Cc)\intertext{and let $W_a$ be represented by an admissible Pro-diagram}
        W_a\colon B_a &\to \nTate(\Cc).
    \end{align*}
    Then we have
    \begin{align*}
        \hom_{(n+1)\text{-}\Tate(\Cc))(\Cc)}(U\rotimes  V,W)&\cong \lim_{I^{\op}}\colim_A\lim_{B_a}\colim_{J_i^{\op}}\hom_{\nTate(\Cc)}(U_{i,j}\rotimes  V,W_{a,b})\\
        &\cong \lim_{I^{\op}}\colim_A\lim_{B_a}\colim_{J_i^{\op}}\hom_{2n\text{-}Tate(\Cc)}(V,\Hom(U_{i,j},W_{a,b}))\\
        &\cong \hom_{(2n+2)\text{-}\Tate(\Cc)}(V,\Hom(U,W)).
    \end{align*}
\epf

\begin{proposition}
    Let $\Cc$ be an idempotent complete rigid tensor category in which the tensor product is bi-exact, i.e. suppose there exists an exact duality
    \begin{align*}
        (-)^\vee\colon\Cc^{\op}&\to^\simeq \Cc\intertext{along with a natural isomorphism} \Hom_{\Cc}(-,-)&\cong (-)^\vee\otimes -.
    \end{align*}
    \begin{enumerate}
        \item For $U\in\Cc$ and $V\in\nTate(\Cc)$, there is a canonical natural isomorphism
            \begin{equation*}
                V\rotimes  U^\vee\cong \Hom(U,V),
            \end{equation*}
            where $V^\vee$ denotes the extension of the duality $(-)^\vee$ to $\nTate(\Cc)$ as in Proposition \ref{prop:duals}.
        \item For $V\in\nTate(\Cc)$ and $U,W\in m\text{-}\Tate(\Cc)$, there is a canonical natural isomorphism
            \begin{equation*}
                \hom_{(n+m)\text{-}\Tate(\Cc)}(U\rotimes  V,W)\cong\hom_{(n+m)\text{-}\Tate(\Cc)}(U,W\rotimes  V^\vee).
            \end{equation*}
    \end{enumerate}
\end{proposition}
\pf
    Both of these statements follow from induction and formal manipulation of limits and colimits. For the first statement, for $n=0$, there is nothing to show. Now suppose we have shown the first statement for $n$. Let $U\in\Cc$ and let $V\in(n+1)\text{-}\Tate(\Cc)$ be represented by an elementary Tate diagram
    \begin{align*}
        V\colon I &\to \Pro(\nTate(\Cc))\intertext{and for each $i\in I$, let $V_i$ be represented by an admissible Pro-diagram}
        V_i\colon J_i&\to \nTate(\Cc).
    \end{align*}
    Then we have
    \begin{align*}
        V\rotimes  U^\vee &\cong \colim_I\lim_{J_i} V_{i,j}\rotimes  U^\vee\\
        &\cong \colim_I\lim_{J_i}\Hom(U,V_{i,j})\\
        &\cong \Hom(U,V).
    \end{align*}
    This completes the induction step and the proof of the first statement.

    For the second statement, the case for $n=0$ and arbitrary $m$ follows from the first statement combined with Proposition \ref{prop:homclosed}. Now suppose we have shown the second statement for $n$ and $m$. Let $U,W\in m\text{-}\Tate(\Cc)$ and let $V\in(n+1)\text{-}\Tate(\Cc)$ be represented by an elementary Tate diagram
    \begin{align*}
        V\colon I &\to \Pro(\nTate(\Cc))\intertext{and for each $i\in I$, let $V_i$ be represented by an admissible Pro-diagram}
        V_i\colon J_i&\to \nTate(\Cc).
    \end{align*}
    Then we have
    \begin{align*}
        \hom_{(n+m+1)\text{-}\Tate(\Cc)}(U\rotimes  V,W)&\cong \lim_{I^{\op}}\colim_{J_i^{\op}}\hom_{(n+m)\text{-}\Tate(\Cc)}(U\rotimes  V_{i,j},W)\\
        &\cong \lim_{I^{\op}}\colim_{J_i^{\op}}\hom_{(n+m)\text{-}\Tate(\Cc)}(U,W\otimes V_{i,j}^\vee)\\
        &\cong \hom_{(n+m+1)\text{-}\Tate(\Cc)}(U,W\otimes V^\vee).
    \end{align*}
    This completes the induction and the proof of the second statement.
\epf

\begin{remark}\mbox{}
    \begin{enumerate}
        \item For $U\in\nTate(\Cc)$, $V\in m\text{-}\Tate(\Cc)$ and $W\in\ell\text{-}\Tate(\Cc)$ with $n\ge m+\ell$, the abelian groups
            \begin{align*}
                &\hom_{(n+m)\text{-}\Tate(\Cc)}(U\rotimes  V,W)\intertext{and}
                &\hom_{\nTate(\Cc)}(U,W\rotimes V^\vee)
            \end{align*}
            are not isomorphic in general. It suffices to take $n=2$, $m=\ell=1$ to see this.
        \item For $U\in\nTate(\Cc)$, $V\in m\text{-}\Tate(\Cc)$ and $W\in\ell\text{-}\Tate(\Cc)$ with $\ell\ge n+m$, the abelian groups
            \begin{align*}
                &\hom_{\ell\text{-}\Tate(\Cc)}(U\rotimes  V,W)\intertext{and}
                &\hom_{(m+\ell)\text{-}Tate(\Cc)}(U,W\rotimes V^\vee)
            \end{align*}
            are not isomorphic in general. Similarly, it suffices to take $n=m=1$ and $\ell=2$ to see this.
    \end{enumerate}
\end{remark}




\section{Formal tubular neighbourhoods}
One of the most prominent applications of Tate categories is the study of the geometry of a variety using adelic methods. In this section, we shall show that the Beilinson ad\`{e}les along a flag of points in a variety can be written as a normally ordered tensor product. In the literature this seems to have gone unnoticed so far.

Recall that $\nTate(k)$ is a shorthand for $\nTate(\mathsf{Vect}_{f}(k))$, the category of $n$-Tate objects over the category of finite-dimensional $k$-vector spaces, and note that because the tensor product of vector spaces is bi-exact, the categories of flat (higher) Tate objects and all (higher) Tate objects coincide.

As we have seen in \S \ref{sect:rightexacttens}, every bi-right exact functor
$\mathcal{C}\times\mathcal{C}\rightarrow\mathcal{C}$ induces a bi-exact
monoidal structure on flat Tate objects, i.e.
\[
-\overrightarrow{\otimes}-:\left.  n\text{-}\mathsf{Tate}^{\flat}%
(\mathcal{C})\right.  \times\left.  m\text{-}\mathsf{Tate}^{\flat}%
(\mathcal{C})\right.  \longrightarrow\left.  (n+m)\text{-}\mathsf{Tate}%
^{\flat}(\mathcal{C})\right.%
\]
and following Waldhausen \cite[\S 1.5]{MR802796} every bi-exact functor
induces a pairing in algebraic $K$-theory:

\begin{definition}
\label{def_ExternalProd}On the level of $K$-theory groups, this takes the form%
\[
\overrightarrow{\otimes}:K_{p}(\left.  n\text{-}\mathsf{Tate}^{\flat
}(\mathcal{C})\right.  )\times K_{q}(\left.  m\text{-}\mathsf{Tate}^{\flat
}(\mathcal{C})\right.  )\longrightarrow K_{p+q}(\left.  (n+m)\text{-}%
\mathsf{Tate}^{\flat}(\mathcal{C})\right.  )\text{.}%
\]
We will also call this product \textquotedblleft$\overrightarrow{\otimes}%
$\textquotedblright\ the \emph{external product} in the $K$-theory of flat
Tate objects.
\end{definition}

\medskip
\textit{Conventions for geometry and ad\`{e}les:} \thinspace All morphisms
of schemes will tacitly be assumed to be separated. For us a \emph{variety} is a scheme of finite type over a field. If $X$ is a scheme and $x,y\in X$ scheme points, we write $x\geq y$ if
$y\in\overline{\{x\}}$, and $x>y$ if additionally $x\neq y$. Write $S(X)_{\bullet}$ for
the simplicial set of flags of points in a scheme, $A(K,\mathcal{F})$ for the
Parshin--Beilinson ad\`{e}les for a subset $K_{\bullet}\subseteq
S(X)_{\bullet}$, and a quasi-coherent sheaf $\mathcal{F}$ (see \cite[\S 2]%
{MR565095}, \cite[\S 2.1]{MR3536437} for details). If $X$ is a purely
$n$-dimensional scheme, an element $\triangle=(\eta_{0}>\cdots>\eta_{n})\in
S(X)_{n}$ with $\operatorname*{codim}_{X}\overline{\{\eta_{i}\}}=i$ will be
called a \emph{saturated} flag.\
\medskip

\begin{definition}
\label{l_Def_LambdaOneTateObjects}Let $X$ be an integral Noetherian scheme. For scheme
points $x,y\in X$, we define a $1$-Tate object in coherent sheaves supported
on the Zariski closure $\overline{\{x\}}$:%
\begin{equation}
\left.  _{x}\Lambda_{y}\right.  :=\left.  \underset{\mathcal{F}\subseteq
\mathcal{O}_{y}}{\operatorname*{colim}}\right.  \left.  \underset{j}{\lim
}\right.  \mathcal{F}/\mathcal{I}_{x}^{j}\in\left.  1\text{-}\mathsf{Tate}%
^{\flat}\right.  (\operatorname*{Coh}\nolimits_{\overline{\{x\}}%
}(X)), \label{lp_1}%
\end{equation}
where $\mathcal{F}$ runs through all coherent sub-sheaves of the
quasi-coherent sheaf $\mathcal{O}_{y}$, $j\in\mathbb{Z}_{\geq1}$, and
$\mathcal{I}_{x}$ denotes the ideal sheaf of $\overline{\{x\}}$. This defines
an exact functor%
\begin{align*}
\left.  _{x}\Omega_{y}\right.  :\operatorname*{Coh}X  & \longrightarrow
\left.  1\text{-}\mathsf{Tate}^{\flat}\right.  (\operatorname*{Coh}%
\nolimits_{\overline{\{x\}}}(X))\\
\mathcal{G}  & \longmapsto\mathcal{G}\overrightarrow{\otimes}\left.
_{x}\Lambda_{y}\right.  \text{,}%
\end{align*}
where $\overrightarrow{\otimes}$ is the tensor product with $\mathcal{G}$,
viewed as a $0$-Tate object. More explicitly, this just means that
\[
\left.  _{x}\Omega_{y}\right.  (\mathcal{G})=\left.  \underset{\mathcal{F}%
\subseteq\mathcal{O}_{y}}{\operatorname*{colim}}\right.  \left.  \underset
{j}{\lim}\right.  (\mathcal{G}\otimes_{\mathcal{O}_{X}}\mathcal{F)}%
/\mathcal{I}_{x}^{j}%
\]
with the same notation as in Equation \ref{lp_1}.
\end{definition}

\pf
Let us justify why this definition makes sense: Firstly, each $\mathcal{F}%
/\mathcal{I}_{x}^{j}$ is a coherent sheaf with support contained in
$\overline{\{x\}}$ since $j\geq1$. Thus, for every coherent sub-sheaf
$\mathcal{F}$ of $\mathcal{O}_{y}$, $\left.  \underset{j}{\lim}\right.
\mathcal{F}/\mathcal{I}_{x}^{j}$ defines an object in $\left.  \mathsf{Pro}%
^{a}\right.  (\operatorname*{Coh}\nolimits_{\overline{\{x\}}}(X))$. By the
Artin--Rees Lemma this object is actually flat in the sense of Definition
\ref{def:flatadmpro}. Hence, $\left.  _{x}\Lambda_{y}\right.  $ is an
Ind-object of flat admissible Pro-objects.
For any scheme point $y$, the morphism $t:\operatorname*{Spec}\mathcal{O}%
_{y}\rightarrow X$ gives rise to a natural morphism $\mathcal{O}%
_{X}\rightarrow t_{\ast}t^{\ast}\mathcal{O}_{X}\cong\mathcal{O}_{y}$. On
stalks this morphism amounts to a localization, i.e. the kernel $\ker
(\mathcal{O}_{X}\rightarrow\mathcal{O}_{y})$ consists only of torsion elements
of the multiplicative system of the localization, and since $X$ is an integral
scheme, no non-zero such can exist. We may therefore regard $\mathcal{O}_{X}$
as a subsheaf of $\mathcal{O}_{y}$.
Then the concrete choice $\mathcal{F}%
:=\mathcal{O}_{X}$ gives rise to the short exact sequence
\[
0\longrightarrow\left.  \underset{j}{\lim}\right.  \mathcal{O}_{X}%
/\mathcal{I}_{x}^{j}\longrightarrow\left.  _{x}\Lambda_{y}\right.
\longrightarrow\left.  \underset{\mathcal{F}\subseteq\mathcal{O}_{y}%
}{\operatorname*{colim}}\right.  \mathcal{F}/\mathcal{O}_{X}\longrightarrow0
\]
in the category of admissible Ind-Pro objects. As the first term is a
Pro-object and the last an Ind-object, it follows that $\left.  _{x}%
\Lambda_{y}\right.  $ is a (flat) Tate object, and thus lies in $\left.
1\text{-}\mathsf{Tate}^{\flat}\right.  (\operatorname*{Coh}%
\nolimits_{\overline{\{x\}}}(X))$. The exactness of the functor $\left.
_{x}\Omega_{y}\right.  $ follows from the flatness of $\mathcal{O}_{y}$ over
$\mathcal{O}_{X}(U)$ for any affine open $U$ containing $y$, and again the
Artin--Rees Lemma.
\epf

\begin{remark}
Following a suggestion of the referee, let us point out that one may think of
this definition as%
\[
\text{\textquotedblleft}\left.  _{x}\Lambda_{y}\right.  \sim\mathcal{O}%
_{X,\widehat{\overline{\{x\}}}\cap y}\text{\textquotedblright,}%
\]
at least philosophically.
\end{remark}

\begin{theorem}Let $k$ be a field. Suppose $X/k$ is a purely
$n$-dimensional integral Noetherian $k$-scheme and $\triangle=(\eta_{0}>\cdots>\eta
_{n})$ a saturated flag. Then the diagram of functors%
\[%
\xymatrix{
\operatorname*{Coh}X \ar[d]^{{\left. _{\eta_{n}}\Omega_{\eta_{n-1}}\right.}}
& \times\cdots\times& \operatorname*{Coh}X \ar[d]_{{\left
. _{\eta_{1}}\Omega_{\eta_{0}}\right.}} \ar[r]^{\otimes}    & \operatorname
*{Coh}X \ar[d]_{A(\triangle,-)} \\
\left.  1\text{-}\mathsf{Tate}^{\flat}(\operatorname*{Coh}\nolimits
_{\overline{\{\eta_{n}\}}}(X))\right.  & \times\cdots\times
& \left.  1\text{-}\mathsf{Tate}^{\flat}(\operatorname*{Coh}\nolimits
_{\overline{\{\eta_{1}\}}}(X))\right.
\ar[r]_{\overrightarrow{\otimes}} & \left.  n\text{-}\mathsf{Tate}^{\flat
}(\operatorname*{Coh}\nolimits_{\overline{\{\eta_{n}\}}}(X))\right.
}%
\]
is commutative. Here ${A(\triangle,-)}$ denotes the Beilinson ad\`{e}les (as in \cite{MR565095}).
\end{theorem}

Before proving this result, we note the most important consequence:\ The Beilinson ad\`{e}les along a flag can functorially be written as
a normally ordered tensor product:

\begin{corollary}\textsc{(Tubular decomposition)}\label{cor:tubedecomp}
With the same assumptions as in the theorem, we have an isomorphism%
\[
A(\triangle,\mathcal{O}_{X})\cong\left.  _{\eta_{n}}\Lambda_{\eta_{n-1}%
}\right.  \overrightarrow{\otimes}\cdots\overrightarrow{\otimes}\left.
_{\eta_{1}}\Lambda_{\eta_{0}}\right.
\]
in the category of $n$-Tate objects of coherent sheaves with support in the
closed point $\eta_{n}$.
\end{corollary}

\pf
We prove the theorem. The horizontal arrow in the top row is the tensor
product of $\mathcal{O}_{X}$-modules. This is a right $n$-polyexact functor,
i.e. for any bracketing $((\ldots)\ldots)$ decomposing it into a concatenation
of $n-1$ functors in two arguments, all these two-argument functors are
bi-right exact. This holds as the individual tensor product%
\[
\operatorname*{Coh}X\times\operatorname*{Coh}X\longrightarrow
\operatorname*{Coh}X
\]
is bi-right exact. The horizontal arrow in the bottom row is the external
product of flat Tate objects, induced from the tensor product of
$\mathcal{O}_{X}$-modules with support,%
\[
\operatorname*{Coh}\nolimits_{Z_{1}}X\times\operatorname*{Coh}\nolimits_{Z_{2}%
}X\longrightarrow\operatorname*{Coh}\nolimits_{Z_{1}\cap Z_{2}}X
\]
for $Z_{1},Z_{2}$ arbitrary closed subsets of $X$. As this is also bi-right
exact, the results of \S \ref{sect:rightexacttens} apply. Thus, all
ingredients of the diagram are well-defined. Next, we need to check
commutativity. Suppose we start from the object $\mathcal{G}_{n-1}\times
\cdots\times\mathcal{G}_{0}$ in the product category in the upper left corner
of the diagram. For $0\leq r<n$ we let%
\[
T^{r}:=\left.  _{\eta_{r+1}}\Lambda_{\eta_{r}}\right.  =\left.  \underset
{\mathcal{F}_{r}\subseteq\mathcal{O}_{\eta_{r}}}{\operatorname*{colim}%
}\right.  \left.  \underset{j_{r+1}}{\lim}\right.  \frac{\mathcal{G}%
_{r}\otimes\mathcal{F}_{r}}{\mathcal{I}_{\eta_{r+1}}^{j_{r+1}}}\in\left.
1\text{-}\mathsf{Tate}^{\flat}(\operatorname*{Coh}\nolimits_{\overline
{\{\eta_{r+1}\}}}(X))\right.  \text{,}%
\]
where $\mathcal{F}_{r}$ runs through the coherent sub-sheaves of
$\mathcal{O}_{\eta_{r}}$. Unravel the construction of the tensor product
following \S \ref{sect:biexacttens}. This yields%
\begin{equation}
T^{n-1}\overrightarrow{\otimes}\cdots\overrightarrow{\otimes}T^{0}=\left.
\underset{\mathcal{F}_{0}\subseteq\mathcal{O}_{\eta_{0}}}%
{\operatorname*{colim}}\right.  \left.  \underset{j_{1}}{\lim}\right.
\cdots\left.  \underset{\mathcal{F}_{n-1}\subseteq\mathcal{O}_{\eta_{n-1}}%
}{\operatorname*{colim}}\right.  \left.  \underset{j_{n}}{\lim}\right.
(\mathcal{G}_{n-1}\otimes\cdots\otimes\mathcal{G}_{0})\otimes\frac
{\mathcal{F}_{n-1}}{\mathcal{I}_{\eta_{n}}^{j_{n}}}\otimes\cdots\otimes
\frac{\mathcal{F}_{0}}{\mathcal{I}_{\eta_{1}}^{j_{1}}}\text{,}\label{lca1}%
\end{equation}
where the tensor products on the right are those of $\mathsf{Mod}%
_{\mathcal{O}_{X}}$. On the other hand, following the inductive definition of
the ad\`{e}les as an $n$-Tate object (see for example \cite{MR3510209} or
\cite{bgwTateModule}), we get%
\begin{align*}
A(\eta_{0}  & >\cdots>\eta_{n},\mathcal{G}_{n-1}\otimes\cdots\otimes
\mathcal{G}_{0})=\left.  \underset{\mathcal{F}_{0}\subseteq\mathcal{O}%
_{\eta_{0}}}{\operatorname*{colim}}\right.  \left.  \underset{j_{1}}{\lim
}\right.  A(\eta_{2}>\cdots>\eta_{n},\mathcal{F}_{0}\otimes\mathcal{O}%
_{\eta_{1}}/\mathcal{I}_{\eta_{1}}^{j_{1}})\\
& =\cdots=\left.  \underset{\mathcal{F}_{0}\subseteq\mathcal{O}_{\eta_{0}}%
}{\operatorname*{colim}}\right.  \left.  \underset{j_{1}}{\lim}\right.
\cdots\left.  \underset{\mathcal{F}_{n-1}\subseteq\mathcal{O}_{\eta_{n-1}}%
}{\operatorname*{colim}}\right.  \left.  \underset{j_{n}}{\lim}\right.
(\mathcal{G}_{n-1}\otimes\cdots\otimes\mathcal{G}_{0})\otimes\frac
{\mathcal{F}_{0}}{\mathcal{I}_{\eta_{1}}^{j_{1}}}\otimes\cdots\otimes
\frac{\mathcal{F}_{n-1}}{\mathcal{I}_{\eta_{n}}^{j_{n}}}\text{,}%
\end{align*}
which is of course literally the same object, thanks to the symmetry of the
tensor product in $\mathcal{O}_{X}$-module sheaves. If $\mathcal{G}%
,\mathcal{G}^{\prime}$ are coherent sheaves with supports in $Z$ and
$Z^{\prime}$, then $\mathcal{G}\otimes\mathcal{G}^{\prime}$ has support in
$Z\cap Z^{\prime}$ and thus this $n$-Tate object actually lies in $\left.
n\text{-}\mathsf{Tate}^{\flat}(\operatorname*{Coh}\nolimits_{W}(X))\right.  $
with $W=\overline{\{\eta_{n}\}}\cap\cdots\cap\overline{\{\eta_{1}\}}%
=\overline{\{\eta_{n}\}}$. While this was just a verification of the
commutativity on the level of objects, all our steps were natural in all
objects $\mathcal{G}_{n-1},\ldots,\mathcal{G}_{0}$, so morphisms between
objects get induced compatibly as well. The corollary follows by evaluation of
the object $\mathcal{O}_{X}\times\cdots\times\mathcal{O}_{X}$ from the upper
left corner in the lower right corner in the two compatible ways.
\epf

\begin{corollary}
\label{cor_ProductCompat}With the same assumptions as in the theorem, the
restriction to vector bundles in the top row%
\[%
\xymatrix{
\operatorname*{VB}X \ar[d]^{{\left. _{\eta_{n}}\Omega_{\eta_{n-1}}\right.}}
& \times\cdots\times& \operatorname*{VB}X \ar[d]_{{\left. _{\eta
_{1}}\Omega_{\eta_{0}}\right.}} \ar[r]^{\otimes}    & \operatorname*{VB}%
X \ar[d]_{A(\triangle,-)} \\
\left.  1\text{-}\mathsf{Tate}^{\flat}(\operatorname*{Coh}\nolimits
_{\overline{\{\eta_{n}\}}}(X))\right.  & \times\cdots\times
& \left.  1\text{-}\mathsf{Tate}^{\flat}(\operatorname*{Coh}\nolimits
_{\overline{\{\eta_{1}\}}}(X))\right.
\ar[r]_{\overrightarrow{\otimes}} & \left.  n\text{-}\mathsf{Tate}^{\flat
}(\operatorname*{Coh}\nolimits_{\overline{\{\eta_{n}\}}}(X))\right.
}%
\]
is a commutative diagram whose horizontal functors are $n$-polyexact and
downward functors exact.
\end{corollary}

\pf
The downward functors are always exact, even for coherent sheaves. The bottom
horizontal arrow is exact as we work with flat Tate objects. The top
horizontal functor now is exact by the local flatness of vector bundles.
\epf

\begin{remark}
\label{rmk_UnitsDefineK1ClassesOfTateObjects}We may also read%
\[
L:=\left.  \underset{\mathcal{F}\subseteq\mathcal{O}_{y}}%
{\operatorname*{colim}}\right.  \left.  \underset{j}{\lim}\right.
\mathcal{F}/\mathcal{I}_{x}^{j}%
\]
as a $k$-algebra \textit{if} we decide to carry out the limit and colimit. We get a
canonical $k$-algebra homomorphism $\mathcal{O}_{y}\rightarrow L$. In
particular, every unit $f\in\mathcal{O}_{y}^{\times}$ acts by multiplication
on $L$. As any such multiplication is compatible with the Ind- and Pro-limit,
it also determines an automorphism of $\left.  _{x}\Lambda_{y}\right.  $ as an
object in $\left.  1\text{-}\mathsf{Tate}^{\flat}\right.  (\operatorname*{Coh}%
\nolimits_{\overline{\{x\}}}(X))$. Finally, every automorphism of an object in
an exact category defines a canonical element in the $K_{1}$-group of this
category. Thus, we get a canonical group homomorphism%
\[
\underset{\circlearrowright\left.  _{x}\Lambda_{y}\right.  }{\left[  -\right]
}:\mathcal{O}_{y}^{\times}\longrightarrow K_{1}(\left.  1\text{-}%
\mathsf{Tate}^{\flat}\right.  (\operatorname*{Coh}\nolimits_{\overline{\{x\}}%
}(X)))\text{.}%
\]

\end{remark}

\begin{definition}
\label{def:UnitActingOnPartialAdeleFactor}In the situation of the previous
remark, write%
\[
\underset{\circlearrowright\left.  _{x}\Lambda_{y}\right.  }{\left[  f\right]
}\in K_{1}(\left.  \mathsf{Tate}^{\flat}(\operatorname*{Coh}%
\nolimits_{\overline{\{x\}}}(X))\right.  )%
\]
for the image of an element $f\in\Oc^\times_y$.
\end{definition}

\section{Ad\`{e}les}\label{sectadeles}

\subsection{Motivation/the classical case}

Let us recall the relation between the degree of a line bundle on a curve and
the ad\`{e}les. To this end, let $k$ be a field, $\pi : X \rightarrow k$ an integral smooth proper curve
and $L$ a line bundle on $X$. Using the Weil uniformization of the Picard
group, the isomorphism class of $L$ has a unique representative in%
\begin{equation}
\operatorname*{Pic}X=\left.  k(X)^{\times}\right\backslash \left.
\mathbf{A}^{\times}\right.  \left/  \mathbf{O}^{\times}\right.  \text{,}%
\label{lcpic}%
\end{equation}
where $\mathbf{A}$ denotes the ad\`{e}les, so that $\mathbf{A}^{\times}$ are
the id\`{e}les, $\mathbf{O}$ the integral ad\`{e}les and $k(X)$ the rational
function field. The multiplicative group $\mathbf{A}^{\times}$ acts on the
ad\`{e}les via multiplication,%
\[
\mathbf{A}^{\times}\circlearrowright\mathbf{A}\text{,}%
\]
but the ad\`{e}les can also be regarded as a $1$-Tate object of finite-dimensional
$k$-vector spaces, and this action induces an automorphism of this $1$-Tate
object. Like any automorphism, this pins down a unique element in the $K_{1}%
$-group of the category. In other words, we get a group homomorphism%
\begin{equation}
\mathbf{A}^{\times}\longrightarrow K_{1}(\left.  \mathsf{Tate}(k)\right.
)\cong\mathbb{Z}\text{.}\label{lcpic2}%
\end{equation}
It is a classical computation that this morphism sends any representative of
the line bundle, as in Equation \ref{lcpic}, to the degree of the line bundle:

\begin{theorem}\textsc{(Weil)} This morphism sends a line bundle $L$ to its degree $\deg L$.
\end{theorem}

\begin{elab}\label{elab_WeilDisrupt}Of course, classically this fact has been formulated without $K$-theory.
Indeed, \eqref{lcpic2} can be made concrete as follows: Choose a splitting
$\mathbf{A}\simeq E\oplus\mathbf{O}$ as $k$-vector spaces. Now, let
$f\in\mathbf{A}^{\times}$ act. Then for any $k$-vector subspace $E^{\prime
}\subseteq\mathbf{A}$ such that $E,fE\subseteq E^{\prime}$ and such that both
subspaces are of finite codimension in $E^{\prime}$, take $\dim(E^{\prime
}/E)-\dim(E^{\prime}/fE)\in\mathbb{Z}$. To be sure that this makes sense, one
has to prove that $E^{\prime}$ exists and the resulting integer is independent
of its choice. The $K$-theoretic approach sweeps these technicalities under
the rug.
\end{elab}

Now, for the present article, we modify the viewpoint: There is an exact push-forward functor%
\[
\pi_{\ast}:\mathsf{Tate}^{\flat}(\operatorname*{Coh}\nolimits_{0}%
(X))\longrightarrow\mathsf{Tate}(k)\text{,}%
\]
where $\operatorname*{Coh}\nolimits_{0}(X)$ are coherent sheaves of
zero-dimensional support. In particular, if $x$ denotes a closed point, all $[f]$ in Definition \ref{def:UnitActingOnPartialAdeleFactor}, taking
values in the $K$-theory of the category $\left.  \mathsf{Tate}^{\flat}(\operatorname*{Coh}%
\nolimits_{\overline{\{x\}}}(X))\right.$ also define classes in the $K$-theory of $\operatorname*{Coh}\nolimits_{0}(X)$. Line bundles are often
given as \v{C}ech cocycle representatives in $H^{1}(X,\mathbb{G}_{m})$. The
corresponding reformulation of \eqref{lcpic2} becomes:

\begin{proposition}\textsc{(Weil - rephrased)}\label{propweildegree}
Let $\pi : X \rightarrow k$ be an integral smooth proper curve with generic point $\eta_{0}$ and
$(f_{\nu\mu})_{\nu\mu}\in H^{1}(X,\mathbb{G}_{m})$ an alternating \v{C}ech representative
of a line bundle $L$ in a finite open cover $\mathfrak{U}=(U_{\alpha})_{\alpha\in I}$, $I$ totally
ordered. For any $x\in X$, let $\alpha(x)$ be the smallest element of $I$ such
that $x\in U_{\alpha(x)}$. Then%
\begin{equation}
\deg(L)=-\sum_{\eta_{1}}\pi_{\ast} \underset{\circlearrowright\left.  _{\eta_{1}}%
\Lambda_{\eta_{0}}\right.  }{\left[  f_{\alpha(\eta_{1})\alpha(\eta_{0}%
)}\right]  }\label{lcpic3}%
\end{equation}
and the sum has only finitely many non-zero summands. The right-hand side defines an element of
$K_{1}(\left.  \mathsf{Tate}(k)\right.  )\cong\mathbb{Z}$, and this integer is
the degree of $L$. Here the sum runs over all closed points $\eta_{1}\in X$.
\end{proposition}

This is the same statement as in Weil's theorem. The sum in \eqref{lcpic3}
corresponds to making the computation in \eqref{lcpic2}
locally at all points of the curve. As our line bundle is given on a concrete
finite open cover, it turns out to suffice to work on these opens. We will not
give a separate proof for the above proposition as it will be the one-dimensional
special case of Theorem \ref{thm_multiplicity}.

Let us discuss how to generalize these ideas to surfaces. To begin, we consider an easy geometric example, and also just locally. In a classical local field, say%
\[
\mathcal{K}_{X,x}\simeq\kappa(x)((t))\qquad\text{for}\qquad\mathcal{K}%
_{X,x}:=\operatorname*{Frac}\mathcal{O}_{X,x}%
\]
on a curve $X$, the local multiplicity of a line bundle is a measure of compatibility between a given $k$-vector space splitting $\mathcal{K}_{X,x}\simeq A\oplus
\mathcal{O}_{X,x}$ and a local id\`{e}le component acting on
it. This corresponds to the valuation of the id\`{e}le at this point. For
example, if%
\[
f_{\alpha(\eta_{1})\alpha(n_{0})}=t^{r}%
\]
for some $r\geq0$ in $\mathcal{K}_{X,x}\simeq\kappa(x)((t))$, then the
multiplicity is%
\[
\dim(\mathcal{O}_{X,x}/t^{r}\mathcal{O}_{X,x})=r\text{.}%
\]
Now suppose $\pi : X \rightarrow k$ is an integral smooth proper surface over a field. Let $L$ and
$L^{\prime}$ be line bundles, given by%
\[
f=(f_{\rho,\nu})_{\rho,\nu\in I}\in H^{1}(X,\mathbb{G}_{m})\qquad
\text{and}\qquad g=(g_{\rho,\nu})_{\rho,\nu\in I}\in H^{1}(X,\mathbb{G}_{m})
\]
in a joint finite open cover $\mathfrak{U}=(U_{\alpha})_{\alpha\in I}$, $I$
totally ordered. The analogous local consideration lets us look at an
ad\`{e}le of the surface. That is, a saturated flag $\triangle=(\eta_{0}%
>\eta_{1}>\eta_{2})$ and its $2$-local field, e.g.,
\[
A(\triangle,\mathcal{O}_{X})\simeq\kappa((s))((t))
\]
for a suitable finite field extension $\kappa/k$. If, say, locally
$f_{\rho,\nu}=t^{r}$ and $g_{\rho,\nu}=s^{m}$ for some $r,m\geq0$, then
clearly we would expect the correct multiplicity to be $r\cdot m$, inspired by
the classical formula for proper intersections%
\[
\dim_{k}\frac{\mathcal{O}_{X,\eta_{2}}}{(t^{r},s^{m})}%
\]
for divisors which are locally cut out by the equations $t^{r}$ and $s^{m}$ in
a neighbourhood of the closed point $\eta_{2}$. Note that this situation gets
more complicated when $r,s\in\mathbb{Z}$, possibly negative, since then
$(t^{r},s^{m})$ might no longer define an ideal in $\mathcal{O}_{X,x}$.
Moreover, there is no reason why we can assume that $f_{\rho,\nu}$ is an
expression only in the variable $t$, and $g_{\rho,\nu}$ only in the variable
$s$. It nonetheless seems plausible that we can make this work in general by
phrasing it in terms of a suitable operation which \textquotedblleft
mixes\textquotedblright\ the disruption of a splitting caused by $f_{\rho,\nu
}$ and the disruption caused by $g_{\rho,\nu}$ simultaneously. This should be
the r\^{o}le of a suitable concept of tensor product -- and indeed of the normally
ordered tensor product in the case at hand. Specifically, for our
surface, we will get the finite sum%
\[
-\sum_{\triangle=(\eta_{0}>\eta_{1}>\eta_{2})}\pi_{\ast } \underset{\circlearrowright
\left.  _{\eta_{2}}\Lambda_{\eta_{1}}\right.  }{\left[  f_{\alpha(\eta
_{2})\alpha(\eta_{1})}\right]  }\left.  \overrightarrow{\otimes}\right.
\underset{\circlearrowright\left.  _{\eta_{1}}\Lambda_{\eta_{0}}\right.
}{\left[  g_{\alpha(\eta_{1})\alpha(\eta_{0})}\right]  }%
\]
in Theorem \ref{thm_multiplicity} below. It has only finitely many non-zero
summands, and defines an element in $K_{2}(\left.  2\text{-}\Tate(k)\right.  )\cong\mathbb{Z}$, and this integer is the intersection
multiplicity $L\cdot L^{\prime}$. The $\left.  \overrightarrow{\otimes
}\right.  $-product denotes the product of $K$-theory classes induced to the
$2$-Tate category. It contains $\left.  _{\eta_{2}}\Lambda_{\eta_{1}}\right.
\left.  \overrightarrow{\otimes}\right.  \left.  _{\eta_{1}}\Lambda_{\eta_{0}%
}\right.  $, which is precisely the $2$-Tate object coming from the ad\`{e}les
along the flag $\triangle$.\medskip

After these introductory comments, let us work out these ideas in a rigorous fashion.

\subsection{Adelic multiplicity formula}

We are ready to prove an adelic variant of an intersection multiplicity formula. It recasts the adelic
intersection pairing of Parshin \cite[\S 2]{MR697316} in a new light.

\begin{theorem}
\label{thm_multiplicity}Let $k$ be a field. Suppose $\pi : X \rightarrow k$ is a purely
$n$-dimensional integral smooth proper variety. Let $L_{1},\ldots,L_{n}$ be line
bundles which are represented by alternating \v{C}ech representatives%
\[
f^{q}=(f_{\rho,\nu}^{q})_{\rho,\nu\in I}\in H^{1}(X,\mathbb{G}_{m})\qquad
\quad\text{(for }q=1,\ldots,n\text{)}%
\]
in a finite open cover $\mathfrak{U}=(U_{\alpha})_{\alpha\in I}$, $I$ totally
ordered. For any $x\in X$, let $\alpha(x)$ be the smallest element of $I$ such
that $x\in U_{\alpha(x)}$. Recall our notation%
\[
\underset{\circlearrowright\left.  _{x}\Lambda_{y}\right.  }{\left[  f\right]
}%
\]
for $K_{1}$-classes from Definition \ref{def:UnitActingOnPartialAdeleFactor}
and the external product in $K$-theory from Definition
\ref{def_ExternalProd}. Using this notation, the sum%
\[
(-1)^{\frac{n(n+1)}{2}}\sum_{\triangle=(\eta_{0}>\cdots>\eta_{n})}%
\pi_{\ast } \underset{\circlearrowright\left.  _{\eta_{n}}\Lambda_{\eta_{n-1}}\right.
}{\left[  f_{\alpha(\eta_{n})\alpha(\eta_{n-1})}^{n}\right]  }\left.
\overrightarrow{\otimes}\right.  \cdots\left.  \overrightarrow{\otimes
}\right.  \underset{\circlearrowright\left.  _{\eta_{1}}\Lambda_{\eta_{0}%
}\right.  }{\left[  f_{\alpha(\eta_{1})\alpha(\eta_{0})}^{1}\right]  }%
\]
has only finitely many non-zero summands, defines an element in $K_{n}(\left.
n\text{-}\mathsf{Tate}(k)\right.  )\cong\mathbb{Z}$, and this integer
is the intersection multiplicity $L_{1}\cdots L_{n}$.
\end{theorem}

We shall address the proof in the next subsection.

\begin{elab}
Let us give a few explanations to make the statement of the formula more
easily digestible:

\begin{enumerate}
\item Each $f_{\rho,\nu}^{\bullet}$ is a function defined on some open $U_{\rho
}\cap U_{\nu}$. It is in particular defined in the local ring underlying
Definition \ref{def:UnitActingOnPartialAdeleFactor}. As a result, we get a
$K_{1}$-class $[f_{\rho,\nu}^{\bullet}]$ in the Tate category of
$\operatorname*{Coh}_{\overline{\{-\}}}(X)$.

\item Then we use the external product in $K$-theory of Definition
\ref{def_ExternalProd}. As a result, we get a class in $K_{n}$ of an $n$-Tate
category. Since the support of coherent sheaves, when tensored, gets
intersected, this will be the $n$-Tate category of coherent sheaves of
support in some zero-dimensional set. In particular, thanks to $\pi_{\ast}$, we
can send this to an $n$-Tate vector space over $k$.

\item Using the canonical isomorphism $K_{n}(\left.  n\text{-}\mathsf{Tate}%
(k)\right.  )\cong\mathbb{Z}$, we can read such a class as an integer and add them.
\end{enumerate}

The reader may wonder about the r\^{o}le of the normally ordered tensor product
behind the scenes here. We allow ourselves to jump ahead a little. Indeed, its r\^{o}le is two-fold: On the one hand,
\textquotedblleft$\overrightarrow{\otimes}$\textquotedblright\ is constructed
from the ordinary tensor product of $\operatorname*{Coh}(X)$, and this is
relevant in the same way as in Proposition
\ref{prop_intersection_via_boundaryformula}. However, secondly, the reader
will notice the absence of a counterpart of the boundary map $\partial_{\bullet
}^{\bullet}$ of Proposition \ref{prop_intersection_via_boundaryformula} in the
formula of Theorem \ref{thm_multiplicity}. This has to do with the second
r\^{o}le of the normally ordered tensor product:\ The underlying object of
this product is precisely the ad\`{e}les $A(\triangle,\mathcal{O}_{X})$ of the
relevant flag. Being realized in precisely this fashion in the $n$-Tate
category basically encodes the datum of the boundary maps of Proposition
\ref{prop_intersection_via_boundaryformula}. Although mathematically
unrelated, this is philosophically analogous to the Uniqueness Principle,
Proposition \ref{prop_UniqueNLocStruct}: Loosely speaking, after completing,
the object is so rigid that the datum of all its valuations is uniquely
encoded in it.
\end{elab}

To prove agreement with the usual
intersection multiplicity, we need to decide which approach for defining the latter
we pick. We choose a viewpoint based on Chow groups. The intersection multiplicity can
be defined by multiplying the divisors of the line bundles in the Chow ring. The latter
is naturally connected to considerations in $K$-theory by the Bloch--Quillen formula.

Let us recall how this works in detail, also in order to set up the notation for our proof
of agreement.

\subsection{Intersection pairing}

Let $X$ be a scheme of finite type over a field $k$. Let $\mathcal{K}_{p}^{\operatorname*{M}}$ denote the $p$-th Milnor
$K$-theory sheaf: We recall its definition. For every Zariski open $U\subseteq
X$, define%
\begin{equation}
\mathcal{K}_{p}^{\operatorname*{M}}\left(  U\right)  :=\ker\left(
{\textstyle\coprod\nolimits_{x\in U^{0}}}
K_{p}^{\operatorname*{M}}\left(  \kappa(x)\right)  \overset{d}{\longrightarrow
}
{\textstyle\coprod\nolimits_{x\in U^{1}}}
K_{p-1}^{\operatorname*{M}}\left(  \kappa(x)\right)  \right)  \text{,}%
\label{lcu1}%
\end{equation}
where $U^{i}$ denotes the set of points $x\in U$ with $\operatorname*{codim}%
_{X}\overline{\{x\}}=i$, $\kappa(x)$ denotes the residue field at $x$ and
$K_{i}^{\operatorname*{M}}(-)$ denotes the $i$-th Milnor $K$-group of a field; we recall $d$ below.
There is a natural restriction map to smaller Zariski opens, making
$\mathcal{K}_{p}^{\operatorname*{M}}$ a sheaf of abelian groups for the
Zariski topology.

See Remark \ref{rmk_AltDefMilnorKSheaf} for an alternative, and perhaps
simpler-looking, definition. If one replaces each occurrence of a Milnor
$K$-group by a Quillen $K$-group, this gives the corresponding definition of
Quillen $K$-theory sheaves $\mathcal{K}_{p}$.

The sheaf $\mathcal{K}_{p}^{\operatorname*{M}}$ has a Zariski resolution, a
quasi-isomorphism to a flasque complex of sheaves concentrated in degrees
$[0,p]$, namely%
\[
\mathcal{K}_{p}^{\operatorname*{M}}(U)\overset{\sim}{\longrightarrow}\left[
{\textstyle\coprod\nolimits_{x\in U^{0}}}
K_{p}^{\operatorname*{M}}\left(  \kappa(x)\right)  \overset{d}{\longrightarrow
}\cdots\overset{d}{\longrightarrow}%
{\textstyle\coprod\nolimits_{x\in U^{j}}}
K_{p-j}^{\operatorname*{M}}\left(  \kappa(x)\right)  \overset{d}%
{\longrightarrow}\cdots\right]  _{0,p}\text{.}
\]
We may call this the \emph{Gersten complex}. In fact, Equation \ref{lcu1} is a definition of the sheaf $\mathcal{K}%
_{p}^{\operatorname*{M}}$ modelled after this resolution, instead of the more
traditional approach of Remark \ref{rmk_AltDefMilnorKSheaf}. The differential
$d$ is of the following shape:%
\[
d=\sum_{x,y}\partial_{y}^{x}\text{,}%
\]
where $x\in U^{j}$ is any codimension $j$ point of the scheme, and $y\in\overline{\{x\}}%
^{1}$ any point of codimension one in the closure $\overline{\{x\}}$. Each
$\partial_{y}^{x}$ in turn is a map%
\[
\partial_{y}^{x}:K_{\ast}^{\operatorname*{M}}(\kappa(x))\longrightarrow
K_{\ast-1}^{\operatorname*{M}}(\kappa(y))\text{,}%
\]
which is defined as follows: The local ring $\mathcal{O}_{\overline{\{x\}},y}$
is one-dimensional since $y$ is of codimension one in $\overline{\{x\}}$. Take
its normalization, i.e. the integral closure $\mathcal{O}_{\overline{\{x\}}%
,y}^{\prime}$ of $\mathcal{O}_{\overline{\{x\}},y}$ inside $\kappa(x)$. As we
work in the context of a variety of finite type over a field, the
normalization is a finite ring extension and as $\mathcal{O}_{\overline
{\{x\}},y}$ is local, this implies that $\mathcal{O}_{\overline{\{x\}}%
,y}^{\prime}$ is a semi-local ring inside $\kappa(x)$. If $\mathfrak{m}%
_{1},\ldots,\mathfrak{m}_{r}$ denotes its maximal ideals, each localization
$(\mathcal{O}_{\overline{\{x\}},y}^{\prime})_{\mathfrak{m}_{i}}$ then is a
discrete valuation ring. Write $v_{i}$ for its normalized valuation. Then
define%
\begin{align}
\partial_{y}^{x}:K_{\ast}^{\operatorname*{M}}(\kappa(x))  & \longrightarrow
K_{\ast-1}^{\operatorname*{M}}(\kappa(y))\label{lBoundMapOfVal}\\
\beta & \longmapsto\sum_{i=1}^{r}N_{\kappa(\mathfrak{m}_{i})/\kappa
(y)}\partial_{v_{i}}(\beta)\text{,}\nonumber
\end{align}
where $\partial_{v}$ is the Milnor $K$-boundary map of the discrete valuation
$v$, and $N_{-/-}$ denotes the norm. The finiteness of the integral closure
also implies that the residue fields $\kappa(\mathfrak{m}_{i})$ of the
valuations $v_{i}$ are finite over $\kappa(y)$.

The same discussion is valid for Quillen $K$-theory (and indeed much more
generally, see \cite{MR1418952}).

Let $\underline{\mathbb{Z}}$ denote the locally constant Zariski sheaf with
values in $\mathbb{Z}$. There is a natural $\mathbb{Z}$-graded commutative
$\underline{\mathbb{Z}}$-algebra structure on the Milnor $K$-sheaves,%
\begin{equation}
\mathcal{K}_{p}^{\operatorname*{M}}\otimes_{\underline{\mathbb{Z}}}%
\mathcal{K}_{q}^{\operatorname*{M}}\longrightarrow\mathcal{K}_{p+q}%
^{\operatorname*{M}}\text{,}\label{lcu2}%
\end{equation}
and analogously on the ordinary $K$-theory sheaves $\mathcal{K}_{p}$. Thus,
these can be promoted to be viewed as sheaves of graded rings $\mathcal{K}%
_{\ast}^{\operatorname*{M}}$ resp. $\mathcal{K}_{\ast}$. Finally, there is a
morphism of sheaves of graded $\underline{\mathbb{Z}}$-algebras%
\begin{equation}
\mathcal{K}_{\ast}^{\operatorname*{M}}\longrightarrow\mathcal{K}_{\ast
}\text{,}\label{lb1}%
\end{equation}
induced from the canonical morphism $K_{p}^{\operatorname*{M}}(F)\rightarrow
K_{p}(F)$ for every field $F$.

\begin{remark}
\label{rmk_AltDefMilnorKSheaf}If the base field $k$ is infinite, there is an
alternative definition for the Milnor $K$-theory sheaf, given by%
\[
\mathcal{K}_{\ast }^{\operatorname*{M}}\left(  U\right)  :=T_{\mathbb{Z}%
}(\mathcal{O}_{X}(U)^{\times})/\left\langle x\otimes\left(  1-x\right)
\mid\text{for all }x\text{ with }x,1-x\in\mathcal{O}_{X}(U)^{\times
}\right\rangle \text{,}%
\]
where $T_{\mathbb{Z}}$ denotes the free tensor algebra as a $\mathbb{Z}%
$-module. This definition is equivalent to the one of Equation \ref{lcu1} for
smooth varieties over infinite fields, as was proven by van der Kallen and
more generally Kerz, see \cite[Proposition 10]{MR2551760}. In fact, it also works
for finite base fields once they have sufficiently large cardinality, cf. loc. cit.
\end{remark}

The cup product of sheaf cohomology, in conjunction with the ring structure of
the Milnor $K$-theory sheaves $\mathcal{K}_{p}^{\operatorname*{M}}%
\otimes_{\mathbb{Z}}\mathcal{K}_{q}^{\operatorname*{M}}\rightarrow
\mathcal{K}_{p+q}^{\operatorname*{M}}$, induces a product%
\[
H^{i}(X,\mathcal{K}_{p}^{\operatorname*{M}})\otimes_{\mathbb{Z}}%
H^{j}(X,\mathcal{K}_{q}^{\operatorname*{M}})\longrightarrow H^{i+j}%
(X,\mathcal{K}_{p}^{\operatorname*{M}}\otimes_{\mathbb{Z}}\mathcal{K}%
_{q}^{\operatorname*{M}})\longrightarrow H^{i+j}(X,\mathcal{K}_{p+q}%
^{\operatorname*{M}})\text{.}%
\]
Applied to $n$ classes in cohomological degree one, this becomes%
\[
H^{1}(X,\mathcal{K}_{1}^{\operatorname*{M}})\otimes_{\mathbb{Z}}\cdots
\otimes_{\mathbb{Z}}H^{1}(X,\mathcal{K}_{1}^{\operatorname*{M}}%
)\longrightarrow H^{n}(X,\mathcal{K}_{n}^{\operatorname*{M}})
\]
and in view of the algebra morphism of \eqref{lb1}, we can map this
compatibly to its counterpart for Quillen $K$-theory.

\subsection{\label{subsect_CechCohom}\v{C}ech cohomology}

Let $X$ be any topological space. Let $\mathfrak{U}=(U_{\alpha})_{\alpha\in
I}$ ($I$ any index set) be an open cover of $X$ and $\mathcal{F}$ a sheaf
of abelian groups. We denote intersections in the given open cover by
\[
U_{\alpha_{0}\ldots\alpha_{r}}:=%
{\textstyle\bigcap\limits_{i=0,\ldots,r}}
U_{\alpha_{i}}\qquad\text{(for }\alpha_{0},\ldots,\alpha_{r}\in I\text{)}.
\]
Then we have the \v{C}ech
cohomology groups, which we denote by $\check{H}^{p}(\mathfrak{U}%
,\mathcal{F})$. They are defined as the cohomology groups of the complex%
\[
\mathrm{\check{C}}^{p}\left(  \mathfrak{U},\mathcal{F}\right)  :=\prod
_{\alpha_{0}\ldots\alpha_{r}\in I^{p+1}}\mathcal{F}\left(  U_{\alpha_{0}%
\ldots\alpha_{r}}\right)  \text{,\qquad\qquad}\delta:\mathrm{\check{C}}%
^{p}\left(  \mathfrak{U},\mathcal{F}\right)  \rightarrow\mathrm{\check{C}%
}^{p+1}\left(  \mathfrak{U},\mathcal{F}\right)  \text{.}%
\]
There is a sub-complex of\emph{\ alternating chains}, requiring additionally
\[
f_{\alpha_{\pi(0)}\ldots\alpha_{\pi(r)}}=\operatorname*{sgn}(\pi
)f_{\alpha_{0}\ldots\alpha_{r}}
\]
for arbitrary permutations $\pi$ on $r+1$
letters, as well as $f_{\alpha_{0}\ldots\alpha_{r}}=0$ whenever any two
indices agree, $\alpha_{i}=\alpha_{j}$ for $i\neq j$. This sub-complex of
alternating chains is quasi-isomorphic to the whole complex. In particular, we
can restrict to working with alternating chains whenever this appears convenient.

\begin{definition}
\label{def_disjoint_decomp}Given an open cover $\mathfrak{U}$ of $X$, a
\emph{disjoint decomposition} consists of pairwise disjoint subsets
$\{(\Sigma_{\alpha})_{\alpha\in I}\mid\Sigma_{\alpha}\subseteq U_{\alpha}\}$
such that $X=%
{\textstyle\bigcup_{\alpha\in I}}
\Sigma_{\alpha}$. Given this datum, and $x\in X$ is a point, $\alpha(x)$
denotes the unique index in $I$ such that $x\in\Sigma_{\alpha(x)}$ holds.
\end{definition}

For all sheaves that we shall be working with, the sheaf cohomology $H^{p}(X,\mathcal{F})$
arises as the colimit of $\check{H}^{p}(\mathfrak{U},\mathcal{F})$ over
refinements of covers.

\begin{proposition}
\label{prop_intersection_via_boundaryformula}Suppose $X/k$ is a purely
$n$-dimensional smooth proper variety and $L_{1},\ldots,L_{n}$ line bundles
which are represented by \v{C}ech representatives%
\[
f^{q}=(f_{\rho,\nu}^{q})_{\rho,\nu\in I}\in H^{1}(X,\mathbb{G}_{m})\qquad
\quad\text{(for }q=1,\ldots,n\text{)}%
\]
in a finite open cover $\mathfrak{U}=(U_{\alpha})_{\alpha\in I}$. Then for
every disjoint decomposition $(\Sigma_{\alpha})_{\alpha\in I}$, the
intersection multiplicity $L_{1}\cdots L_{n}$ equals%
\[
\sum_{\eta_{n},\ldots,\eta_{0}}[\kappa(\eta_{n}):k]\cdot\partial_{\eta_{n}%
}^{\eta_{n-1}}\cdots\partial_{\eta_{1}}^{\eta^{0}}\{f_{\alpha(\eta_{0}%
)\alpha(\eta_{1})}^{1},f_{\alpha(\eta_{1})\alpha(\eta_{2})}^{2},\ldots
,f_{\alpha(\eta_{n-1})\alpha(\eta_{n})}^{n}\}\in\mathbb{Z}\text{,}%
\]
where the sum runs over all chains of points $\eta_{0}>\cdots>\eta_{n}$.
\end{proposition}

\pf
The degree one case of the Bloch--Quillen formula is the classical isomorphism
$H^{1}(X,\mathbb{G}_{m})\cong\operatorname*{CH}^{1}(X)$, identifying
isomorphism classes of line bundles with Weil divisor classes. The
intersection multiplicity $L_{1}\cdots L_{n}$ is, by definition, the
push-forward of the product of the Weil divisor classes in the Chow ring to the
base field, i.e. for the proper structure morphism $s:X\rightarrow
\operatorname*{Spec}k$,
\begin{equation}
L_{1}\cdots L_{n}=s_{\ast}([L_{1}]\smile\cdots\smile\lbrack L_{n}%
])\text{.}\label{lcu3}%
\end{equation}
The Bloch--Quillen formula is compatible with the intersection product,%
\[
\xymatrix@R=0.16in{ \limfunc{CH}\nolimits^{p}(X) \otimes_{\mathbb{Z}}
\limfunc{CH}\nolimits^{q}(X) \ar[r] \ar[d] & \limfunc{CH}\nolimits
^{p+q}(X) \ar[d] \\ H^{p}(X,\mathcal{K}_{p}^{\limfunc{M}}) \otimes_{\mathbb
{Z}} H^{q}(X,\mathcal{K}_{q}^{\limfunc{M}}) \ar[r] & H^{p+q}(X,\mathcal
{K}_{p+q}^{\limfunc{M}}) }%
\]
and so we may equivalently evaluate the sheaf cup product of the $H^{1}%
$-classes in $H^{1}(X,\mathbb{G}_{m})$, giving a cohomology class in
$H^{n}(X,\mathcal{K}_{n}^{\operatorname*{M}})$. Finally, the isomorphism
$H^{n}(X,\mathcal{K}_{n}^{\operatorname*{M}})\overset{\sim}{\rightarrow
}\operatorname*{CH}^{n}(X)$, or equivalently to $\operatorname*{CH}_{0}(X)$,
is made explicit by \cite[Theorem 2]{MR3104562} and yields%
\[
h_{\eta_{n}}=\sum_{\eta_{n-1},\ldots,\eta_{0}}\partial_{\eta_{n}}^{\eta_{n-1}%
}\cdots\partial_{\eta_{1}}^{\eta_{0}}\{f_{\alpha(\eta_{0})\alpha(\eta_{1}%
)}^{1},f_{\alpha(\eta_{1})\alpha(\eta_{2})}^{2},\ldots,f_{\alpha(\eta
_{n-1})\alpha(\eta_{n})}^{n}\}
\]
for the zero cycle $h=\sum_{\eta_{n}\in X^{n}}h_{\eta_{n}}[\eta_{n}]$, which
is the output of the aforementioned isomorphism (the symbol $\partial_{\ast
}^{\ast}$ denotes the boundary map in Milnor $K$-theory, see \textit{loc. cit.
}or Equation \ref{lBoundMapOfVal}).\ Here the sum runs over all chains $\eta_{0}%
>\cdots>\eta_{n-1}$ and $\alpha$ is as in the statement of our claim. Note
that the set of codimension $n$ points $X^{n}$ is the same as the set of
closed points, given that $X$ is of pure dimension $n$. Returning to Equation
\ref{lcu3}, it thus remains to compute the push-forward $s_{\ast}(h)$, which
amounts to the formula of our claim, and the additional sum index $\eta_{n}$
stems from the summation over all closed points which appear in the evaluation
of $s_{\ast}$, \cite[\S 3]{MR1418952}. This finishes the proof.
\epf

\subsection{Proof of the theorem}

Now we address the proof of Theorem \ref{thm_multiplicity}.

\subsubsection{Step 1: Higher valuations\ (algebraic)}

We shall use the abbreviation DVF\ for discrete valuation fields.

\begin{definition}
Let $k$ be a field.

\begin{enumerate}
\item We call the field $k$ a $0$\emph{-DVF} \emph{with last residue field
}$k$.

\item We call a field $F$ an $n$\emph{-DVF with last residue field} $k$ if it
is a DVF such that its residue field comes equipped with the structure of an
$(n-1)$-DVF with last residue field $k$. Write $\mathcal{O}_{1}$ for the ring
of integers of $F$.
\end{enumerate}
We may write $F$ as a pair $(F,\mathcal{E})$, where $\mathcal{E}$ refers to
the datum of the $n$-DVF\ structure, if we want to stress the choice of the latter.
\end{definition}

Given this datum, it inductively determines a sequence of fields and their
rings of integers. We get the following diagram%
\begin{equation}%
\bfig\node x(0,1200)[F]
\node y(0,900)[\mathcal{O}_{1}]
\node z(300,900)[k_1]
\node w(300,600)[\mathcal{O}_{2}]
\node u(600,600)[k_2]
\node v(600,300)[\vdots,]
\arrow/{^{(}->}/[y`x;]
\arrow/{->>}/[y`z;]
\arrow/{^{(}->}/[w`z;]
\arrow/{->>}/[w`u;]
\arrow[v`u;]
\efig
\label{lcg1}%
\end{equation}
where $\mathcal{O}_{i}$ denotes the rings of integers, and $k_{i}$ the
respective residue fields. Each upward arrow denotes the inclusion of the ring
of integers into its field of fractions, whereas each rightward arrow denotes
a quotient map $\mathcal{O}_{i}\twoheadrightarrow\mathcal{O}_{i}%
/\mathfrak{m}_{i}$, where $\mathfrak{m}_{i}$ denotes the maximal ideal. So all
the rings of integers and these maps canonically come with $\mathcal{E} $.

\begin{definition}
Let $k$ be a field.

\begin{enumerate}
\item We call the field $k$ a $0$\emph{-local field with last residue field} $k$.

\item We call a field $F$ an $n$\emph{-local field with last residue field} $k$ if it
is a complete discrete valuation field such that its residue field comes
equipped with the structure of an $(n-1)$-local field.
\end{enumerate}
\end{definition}

In particular, an $n$-local field is a particular example of an $n$-DVF. The
only difference in the definition is that an $n$-local field is complete with
respect to the topology induced from the discrete valuation, whereas there is
no such requirement for an $n$-DVF. The contrast is quite strong however. We
recall a classical fact due to Parshin:

\begin{proposition}\textsc{(Uniqueness Principle)} \label{prop_UniqueNLocStruct}If a (not necessarily
topologized) field $F$ admits the structure of an $n$-local field with last
residue field $k$, then this structure is unique. Call it $\mathcal{E}%
_{\operatorname*{local}}$.
\end{proposition}

For a proof see \cite[Cor. 1.3, (a)]{MR3536437}. In particular, it is not necessary
to specify the higher local field structures on the residue fields $k_{i}$ manually.

\begin{example}
There is a field inclusion $k(s,t)\hookrightarrow k((s))((t))$. The right
field is a $2$-local field with last residue field $k$; concretely Figure
\ref{lcg1} takes the shape%
\[%
%
%
\bfig\node x(0,1200)[{k((s))((t))}]
\node y(0,900)[{k((s))[[t]]}]
\node z(800,900)[k{((s))}]
\node w(800,600)[{k[[s]]}]
\node u(1600,600)[k]
\arrow/{^{(}->}/[y`x;]
\arrow/{->>}/[y`z;]
\arrow/{^{(}->}/[w`z;]
\arrow/{->>}/[w`u;]
\efig
\]
Step by step, we may restrict the discrete valuations to the field $k(s,t)$
and its residue fields, giving $k(s,t)$ the structure of a $2$-DVF. There is a
similar field inclusion $k(s,t)\hookrightarrow k((t))((s))$. Proceeding
analogously, this equips $k(s,t)$ with two \emph{different} $2$-DVF
structures. On the contrary $k((s))((t))$ only carries the $t$-adic valuation
and $k((t))((s))$ only the $s$-adic. Indeed, the sequence $x_{n}:=s^{-n}t^{n}$
in $k(s,t)$ converges to zero in $k((s))((t))$, but diverges in $k((t))((s))$.
\end{example}

For each $n$-DVF $(F,\mathcal{E})$, there is a \emph{boundary map}
$\partial_{v}$ in $K$-theory%
\[
\partial_{v}:K_{\ast}(F)\longrightarrow K_{\ast-1}(k_{1})\text{.}%
\]
The map $\partial_{v}$ can be constructed as follows: The open-closed
complement decomposition of $\operatorname*{Spec}\mathcal{O}_{1}$ into its
generic and special point gives rise to a fiber sequence in $K$-theory, the
localization sequence%
\begin{equation}
\operatorname*{Spec}\mathcal{O}_{1}/\mathfrak{m}_{1}\hookrightarrow
\operatorname*{Spec}\mathcal{O}_{1}\hookleftarrow\operatorname*{Spec}%
F\text{,}\qquad K(k_{1})\longrightarrow K(\mathcal{O}_{1})\longrightarrow
K(F)\text{,}\label{lcg4}%
\end{equation}
and then $\partial_{v}$ denotes the connecting homomorphism of the long exact
sequence of homotopy groups associated to a fiber sequence.

\begin{remark}
Under the map from Milnor to Quillen $K$-theory, this map $\partial_{v}$ is
compatible with the corresponding boundary map in Milnor $K$-theory,
\cite{MR1418952}.
\end{remark}

As we have just explained, we get a map $K_{\ast}(F)\longrightarrow K_{\ast
-1}(k_{1})$ for each $n$-DVF $F$, but since $k_{1}$ is by definition an
$(n-1)$-DVF, we may inductively concatenate these boundary maps all down to
the last residue field, `jumping down' all the steps of the staircase in Figure \ref{lcg1}.

\begin{definition}
\label{def_AlgHigherVal}Let $(F,\mathcal{E})$ be an $n$-DVF with last residue
field $k_{n}$ and such that $k_{n}$ is a finite extension of $k$. We call the
morphism%
\[
V_{F,\mathcal{E}}:K_{\ast}(F)\overset{\partial}{\longrightarrow}K_{\ast
-1}(k_{1})\overset{\partial}{\longrightarrow}\cdots\overset{\partial
}{\longrightarrow}K_{\ast-n}(k_{n})\overset{N_{k_{n}/k}}{\rightarrow}%
K_{\ast-n}(k)
\]
the \emph{(algebraic) higher valuation} of $(F,\mathcal{E})$. Here $N_{k_{n}/k}$ denotes the
norm map along the field extension $k_{n}/k$.
\end{definition}

\subsubsection{Step 2: Higher valuations (via Ind-Pro methods)}

Before we proceed, we need to recall the fundamental delooping result of
Sho\ Saito \cite{MR3317759}. There is a canonical equivalence%
\begin{equation}
D:\Omega K(\left.  n\text{-}\mathsf{Tate}(\mathcal{C})\right.
)\longrightarrow K(\left.  (n-1)\text{-}\mathsf{Tate}(\mathcal{C})\right.
)\label{lcg3}%
\end{equation}
for idempotent complete exact categories $\mathcal{C}$ (where $\Omega$ denotes
the desuspension of a spectrum, i.e. the degree shift inverse to the
suspension). Such an equivalence was first constructed in loc. cit. by Saito,
but for us it is most convenient to define it by $D:=\mathsf{Index}$, with the
latter constructed as in \cite{bgwRelativeTateObjects} \footnote{In
\cite{IndexMapAlgKTheory} it is proven that this equivalence is essentially
the same as Saito's equivalence, so we favour using the map of
\cite{bgwRelativeTateObjects} to avoid logically relying on this comparison
result.\medskip}.

\begin{example}
\label{ex_IndexMapComp}Naturally, the Laurent series $k((s))$ can be
interpreted as an object in $\left.  1\text{-}\mathsf{Tate}(k)\right.  $.
Moreover, every automorphism of an object in an exact category $\mathcal{C}$
induces a canonical element in $K_{1}(\mathcal{C})$. Thus, multiplication by a
series $f\in k((s))$ defines an element $[f]\in K_{1}(\left.  1\text{-}%
\mathsf{Tate}(k)\right.  )$. Using the explicit description of
\cite{IndexMapAlgKTheory} the value of $D([f])\in K_{0}(k)$ is computed as
follows: Pick two lattices $L_{1},L_{2}\hookrightarrow k((s))$ such that%
\[
L_{1}\subseteq L_{2}\qquad\text{and}\qquad fL_{1}\subseteq L_{2}%
\]
and then take the $K_{0}$-class of $[L_{2}/f_{1}L_{1}]-[L_{2}/L_{1}]\in
K_{0}(k)$. Since $K_{0}(k)\cong\mathbb{Z}$, and this isomorphism is just the
dimension of the underlying vector space, we get%
\[
\lbrack L_{2}/f_{1}L_{1}]-[L_{2}/L_{1}]=v(f)\cdot\lbrack k]\text{,}%
\]
where $v(f)\in\mathbb{Z}$ is the $s$-adic valuation, and $[k]$ the class of a
one-dimensional $k$-vector space. Compare this to Elaboration \ref{elab_WeilDisrupt}.
\end{example}

We turn to geometry. Let $k$ be a field and $\pi : X \rightarrow k$ an
integral variety of pure dimension $n$. Let us write $F:=k(X)$ for its
rational function field. If $\triangle:=\{(\eta_{0}>\cdots>\eta_{n})\}\in
S(X)_{n}$ denotes a saturated flag, we have the Parshin--Beilinson ad\`{e}le
ring%
\[
A(\triangle,\mathcal{O}_{X})=A(\{\eta_{0}>\cdots>\eta_{n}\},\mathcal{O}%
_{X})\text{.}%
\]
We may regard this solely as a ring. Then there is a canonical isomorphism of
$k$-algebras%
\[
A(\triangle,\mathcal{O}_{X})\cong\prod_{i=1}^{r}F_{i}%
\]
for some finite $r$, and each $F_{i}$ is an $n$-local field with last residue
field a finite extension of $k$. For a proof, see \cite[\S 3]{MR1213064} or
\cite[Theorem 4.2]{MR3536437}. Moreover, there is a natural ring homomorphism%
\begin{equation}
\operatorname*{diag}:F\longrightarrow A(\triangle,\mathcal{O}_{X}%
)\text{,}\label{lcg7}%
\end{equation}
frequently called the \emph{diagonal embedding}.

In particular, each $F_{i}$ comes equipped with a canonical algebraic higher
valuation as in Definition \ref{def_AlgHigherVal}, using the single and unique
$n$-DVF structure $\mathcal{E}_{\operatorname*{local}}$, Proposition
\ref{prop_UniqueNLocStruct}. Thus, we obtain a canonical homorphism%
\[
V_{alg}(\triangle):K_{n}(F)\overset{\operatorname*{diag}}{\longrightarrow
}\prod_{i=1}^{r}K_{n}(F_{i})\overset{V_{F_{i},\mathcal{E}%
_{\operatorname*{local}}}}{\longrightarrow}K_{0}(k)\text{.}%
\]

On the other hand, instead of viewing the ad\`{e}les as a ring, they also
carry the structure of an $n$-Tate object in finite-dimensional $k$-vector
spaces, \cite[Theorem 7.10]{MR3510209}, alongside an exact functor%
\[
\operatorname*{real}:\operatorname*{Coh}(X)\longrightarrow\left.
n\text{-}\mathsf{Tate}^{el}\right.  (\operatorname*{Coh}\nolimits_{0}%
(X))\text{,}%
\]
where $\operatorname*{Coh}\nolimits_{0}(X)$ denotes the abelian category of
coherent sheaves of zero-dimensional support. This functor is known as the
\emph{Tate realization}.

Once the support is zero-dimensional, it is finite over the target under the
structure map $\pi : X\rightarrow k$. Hence, the push-forward is
an exact functor $\pi_{\ast }: \operatorname*{Coh}\nolimits_{0}(X)\rightarrow
\mathsf{Vect}_{f}(k)$, and by functoriality of $n$-Tate categories,
\cite[Theorem 7.2, (3)]{MR3510209}, we obtain the exact functor composed as%
\[
\operatorname*{Coh}(X)\overset{{\operatorname{real}}}{\longrightarrow} \left.  n\text{-}\mathsf{Tate}%
^{el}\right.  (\operatorname*{Coh}\nolimits_{0}(X))\overset{{\pi_{\ast }}}{\longrightarrow} \left.
n\text{-}\mathsf{Tate}^{el}\right.  (k)\text{.}%
\]
This remains true if we replace $X$ by a smaller Zariski open neighbourhood of
the generic point of $X$, so that we may take the colimit. This only affects
the source of the functor and we obtain an exact functor $\mathsf{Vect}%
_{f}(F)\rightarrow\left.  n\text{-}\mathsf{Tate}^{el}\right.  (\mathsf{Vect}%
_{f}(k))$, as the function field $F$ is the stalk at the generic point.

This leads to a second map%
\begin{equation}
V_{Tate}(\triangle):K_{n}(F)\overset{\operatorname*{real}}{\longrightarrow
}K_{n}(\left.  n\text{-}\mathsf{Tate}^{el}\right.  (k))\overset{D\circ
\cdots\circ D}{\longrightarrow}K_{0}(k)\text{,}\label{eqVTate}%
\end{equation}
where $D$ denotes the equivalence of Equation \ref{lcg3}.

The following holds for all saturated flags $\triangle\in S(X)_{n}$.

\begin{lemma}
\label{vtate_vs_valg}We have $V_{Tate}(\triangle)=V_{alg}(\triangle)$.
\end{lemma}

\pf
Note that the definition of $V_{alg}$ essentially hinges on the concatenation
of $n$ boundary maps, while the definition of $V_{Tate}$ uses a concatenation
of $n$ index maps $D$. Thus, we will show the compatibility of a single
boundary map with a single index map. Once this is done, use this argument
inductively for all $n$ steps. Concretely: in the case at hand, if $F^{\prime
}$ is a $d$-local field with ring of integers $\mathcal{O}_{1}$, use
\cite[Theorem 1.1]{bgwRelativeTateObjects} with the choice
$X:=\operatorname*{Spec}\mathcal{O}_{1}$, and $Z$ the closed sub-scheme defined
by the special point, i.e., $Z:=\operatorname*{Spec}\mathcal{O}_{1}%
/\mathfrak{m}_{1}$. Then the open complement is $U:=\operatorname*{Spec}%
F^{\prime}$, so we are in the situation of Equation \ref{lcg4}. The functor
$\mathbf{T}_{Z}$ of loc. cit. is compatible with the Tate realization functor
of the ad\`{e}les, and the index equivalence $\mathsf{i}$ of loc. cit. relies
on the map $\mathsf{Index}$ of loc. cit., which is what we have taken as our
definition of the map $D$.
\epf

\begin{remark}
Example \ref{ex_IndexMapComp} demonstrates the previous lemma in the simplest
possible case:\ The evaluation of $D$ yields the same as the boundary map
$\partial_{1}:K_{1}(k((s)))\rightarrow K_{0}(k)$, since the latter is nothing
but the valuation map.
\end{remark}

\subsubsection{Step 3: Comparison with geometrically defined multiplicity}

Next, we will compare the algebraic higher valuation with the maps
$\partial_{y}^{x}$ which appear in the Gersten complex:

We define%
\[
V_{Ger}:K_{n}(F)\longrightarrow K_{0}(k)
\]
by the formula%

\[
\beta\longmapsto\sum_{\triangle=(\eta_{n},\ldots,\eta_{0})}[\kappa(\eta
_{n}):k]\cdot\partial_{\eta_{n}}^{\eta_{n-1}}\cdots\partial_{\eta_{1}}%
^{\eta_{0}}\beta\text{,}%
\]
where $\partial_{\ast}^{\ast}$ is defined as in the description of the
differential of the Gersten complex, i.e., as in Equation \ref{lBoundMapOfVal}%
. The sum runs over all saturated flags $\triangle\in S(X)_{n}$. (It is easy
to see that all but finitely many summands are zero, so this is well-defined)

Note that since $X$ is integral, it has only a single generic point, so we
must always have $\overline{\{\eta_{0}\}}=X$ for all flags $\triangle$, and
correspondingly $\kappa(\eta_{0})=F$.

To bridge between $V_{alg}$ and $V_{Ger}$, we need to recall two standard
facts regarding the functoriality of $\partial_{v}$:\medskip

Firstly, the push-forward compatibility of the boundary map with finite field
extensions: Suppose $L$ is a discrete valuation field with valuation $v$ and
residue field $\kappa(v)$. Let $L^{\prime}/L$ be a finite extension, and
$A^{\prime}$ the integral closure of the valuation ring of $L$ inside
$L^{\prime}$. Suppose $A^{\prime}$ is a finite $A$-module (this holds for
example if $A$ is excellent). Then $A^{\prime}$ is semi-local and if $w$ runs
through the discrete valuations extending $v$, the diagram%
\begin{equation}%
\xymatrix{
K_{\ast}(L^{\prime}) \ar[r]^-{\oplus{\partial}_w} \ar[d]_{N_{L^{\prime} /L}}
& \bigoplus_{w} K_{\ast- 1}(\kappa(w)) \ar[d]^{\Sigma N_{\kappa(w) / \kappa
(v)}} \\
K_{\ast}(L) \ar[r]_-{{\partial}_v} & K_{\ast- 1}(\kappa(v))
}%
\label{lNormSquare}%
\end{equation}
commutes (cf. \cite[R3b]{MR1418952}).

Secondly, there is also a pullback compatibility of the boundary map along
arbitrary field extensions: Suppose $L$ is a discrete valuation field with
valuation $v$ and $L^{\prime}/L$ an arbitrary field extension; write
$i:L\hookrightarrow L^{\prime}$ for the inclusion. Suppose $w$ is a discrete
valuation of $L^{\prime}$ extending $v$ with ramification index $e$. Then%
\begin{equation}%
\xymatrix{
K_{\ast}(L^{\prime}) \ar[r]^-{{\partial}_w} & K_{\ast - 1}(\kappa(w)) \\
K_{\ast}(L) \ar[u]_{i} \ar[r]_-{e \cdot{\partial}_v} & K_{\ast - 1}(\kappa
(v)) \ar[u]
}%
\label{lCorSquare}%
\end{equation}
commutes, and the upward arrow on the right is induced from $i$ by the fact that the
valuation $w$ extends $v$ (cf. \cite[R3a]{MR1418952}).

\begin{lemma}
\label{vger_vs_sum_valg}We have $V_{Ger}=\sum_{\triangle\in S(X)_{n}}%
V_{alg}(\triangle)$.
\end{lemma}

\pf
(Step 1) Note that the above map $V_{Ger}$ is a concatenation of boundary maps%
\[
\partial_{y}^{x}:K_{\ast}(\kappa(x))\longrightarrow K_{\ast-1}(\kappa
(y))\text{,}%
\]
where $x,y$ are points of the scheme, and $\kappa(-)$ the respective residue
fields.
 Unravelling each summand in%
\begin{align}
\partial_{y}^{x}:K_{\ast}(\kappa(x))  & \longrightarrow
K_{\ast-1}(\kappa(y))\nonumber\\
\beta & \longmapsto\sum_{i=1}^{r}N_{\kappa(\mathfrak{m}_{i})/\kappa
(y)}\partial_{v_{i}}(\beta)\label{lch4}%
\end{align}
individually, and using the norm compatibility of Diagram \ref{lNormSquare} to
move all norms jointly all to the left, we can rewrite%
\begin{align}
& \partial_{\eta_{n}}^{\eta_{n-1}}\cdots\partial_{\eta_{1}}^{\eta_{0}}%
\beta\label{lch5}\\
& \qquad\qquad=\sum_{(v_{1},\ldots,v_{n})}N_{\kappa(v_{n})/k(\eta_{n}%
)}\partial_{v_{n}}\cdots\partial_{v_{2}}\partial_{v_{1}}\beta\text{,}\nonumber
\end{align}
where $v_{1},\ldots,v_{n}$ are the discrete valuations which stem from the
individual valuations of the maximal ideals $\mathfrak{m}_{i}\subseteq
\kappa(x)$ for $x=\eta_{0},\ldots,\eta_{n-1}$ when unravelling $\partial
_{\eta_{n}}^{\eta_{n-1}}\cdots\partial_{\eta_{1}}^{\eta_{0}}$ using Formula
\ref{lch4} (it may be helpful to compare this to the explicit explanation and
description of $\partial_{\ast}^{\ast}$ which we have given around Equation
\ref{lBoundMapOfVal}). These discrete valuations $v_{1},\ldots,v_{n}$ actually
define the structure of an $n$-DVF on the function field $\kappa(\eta_{0})=F$,
since $\eta_{0}$ is the single generic point of $X$. We write
$\mathcal{E}_{v_{1},\ldots,v_{n}}$ to denote this $n$-DVF structure with last
residue field $\kappa(v_{n})$.\newline(Step 2) Now compare the right-hand side of
Equation \ref{lch5} with Definition \ref{def_AlgHigherVal}. As norms along finite field
extensions and the boundary map at a valuation $\partial_{v}$ are compatible
(Diagram \ref{lNormSquare}), we get%
\begin{align}
V_{Ger}:K_{n}(F)  & \longrightarrow K_{0}(k)\text{,}\label{lcips2}\\
\beta & \longmapsto\sum_{\triangle}\sum_{v_{1},\ldots,v_{n}}V_{F,\mathcal{E}%
_{v_{1},\ldots,v_{n}}}\nonumber
\end{align}
relying on Definition \ref{def_AlgHigherVal}. Here the sum runs over (1) all
summands $\triangle$ in the definition of $V_{Ger}$, and (2) over all the
valuation chains $v_{1},\ldots,v_{n}$, which arise from the finitely many
choices of maximal ideals $\mathfrak{m}_{\ast}$ in the individual unravelling
steps of Equation \ref{lch4} (for that one fixed $\triangle$ of the first sum).

Note that the only difference with $V_{alg}$, i.e.,%
\[
V_{alg}(\triangle):K_{n}(F)\overset{\operatorname*{diag}}{\longrightarrow
}\prod_{i=1}^{r}K_{n}(F_{i})\overset{V_{F_{i},\mathcal{E}%
_{\operatorname*{local}}}}{\longrightarrow}K_{0}(k)\text{,}%
\]
is that instead of using $F$ and the $n$-DVF structure $\mathcal{E}$, the map
$V_{alg}(\triangle)$ first goes through the field extension $F\hookrightarrow
F_{i}$, and then uses the (canonical) $n$-DVF\ structure $\mathcal{E}%
_{\operatorname*{local}}$ on the latter. The valuations of $F_{i}$ are
precisely the ones lying over the valuations $v_{1},\ldots,v_{n}$ of
$\mathcal{E}_{v_{1},\ldots,v_{n}}$ (actually: given the Uniqueness Principle,
Proposition \ref{prop_UniqueNLocStruct}, $F_{i}$ only has one single $n$-local field
structure, so we cannot go wrong here anyway). However, since the boundary
maps along valuations are compatible with field extensions (Diagram
\ref{lCorSquare}), both results agree. We merely factor the evaluation of a
concatenated boundary map over $n$-local field completions. Thus, Equation
\ref{lcips2} transforms into%
\[
\beta\longmapsto\sum_{\triangle}V_{alg}(\triangle)\text{,}%
\]
confirming our claim.
\epf

\subsubsection{Step 4: End of the proof}

Now we can prove our formula for the intersection multiplicity.

\pf
[of Theorem \ref{thm_multiplicity}]As the cover is indexed by the totally ordered set $I$, we obtain a disjoint decomposition $({\Sigma_{\alpha})}$ in the sense of Definition \ref{def_disjoint_decomp} by letting ${\Sigma_{\alpha}}$ be the set of those scheme points $x\in U_{\alpha}$ such that for no strictly smaller $\alpha^{\prime}<\alpha$ we have $x\in U_{{\alpha}^{\prime }}$.

Consider the adelic expression which
we aim to evaluate. Making the tacit identification $K_{n}(\left.
n\text{-}\mathsf{Tate}(k)\right.  )\cong\mathbb{Z}$ explicit, it reads:
\begin{align}
&  (-1)^{\frac{n(n+1)}{2}}\sum_{\triangle=(\eta_{n}<\cdots<\eta_{0})}%
\pi_{\ast } (D\circ\cdots\circ D)_{A_{X}(\triangle,\mathcal{O}_{X})}\nonumber\\
&  \qquad\qquad\qquad\qquad\left(  \underset{\circlearrowright\left.
_{\eta_{n}}\Lambda_{\eta_{n-1}}\right.  }{\left[  f_{\alpha(\eta_{n}%
)\alpha(\eta_{n-1})}^{n}\right]  }\left.  \overrightarrow{\otimes}\right.
\cdots\left.  \overrightarrow{\otimes}\right.  \underset{\circlearrowright
\left.  _{\eta_{1}}\Lambda_{\eta_{0}}\right.  }{\left[  f_{\alpha(\eta
_{1})\alpha(\eta_{0})}^{1}\right]  }\right)  \text{.}\label{lcem4}%
\end{align}
Firstly, note that we are using the external product of $K$-theory classes
$K_{1}\times\cdots\times K_{1}\rightarrow K_{n}$ here, as introduced in
Definition \ref{def_ExternalProd}. By Corollary \ref{cor_ProductCompat}
this product commutes with taking the product in the ordinary $K$-theory of
the variety $X$ first, and then performing the Tate realization. We use that
this is not just true for $X$, but also holds on all its open subsets; in
particular on the opens $U_{\ast}$ where the $f_{\ast,\ast}^{\ast}$ of the
open cover $\mathfrak{U}$ are defined. As%
\[
\underset{\circlearrowright(-)}{\left[  f\right]  }%
\]
denotes the Tate realization, we may replace the term in the second line,
Equation \ref{lcem4}, by%
\[
\operatorname*{diag}\nolimits^{\ast}\{f_{\alpha(\eta_{n})\alpha(\eta_{n-1}%
)}^{n},\ldots,f_{\alpha(\eta_{1})\alpha(\eta_{0})}^{1}\}\text{,}%
\]
where $\{-,\cdots,-\}$ now refers to the product in the $K$-theory of the open
$U_{\ast}$ in $\mathfrak{U}$, where the $f_{\ast,\ast}^{\ast}$ are defined. As
alternating \v{C}ech cocycles satisfy $f_{\nu\mu}=-f_{\mu\nu}$ (or $f_{\nu\mu
}=f_{\mu\nu}^{-1}$, written for $\mathbb{G}_{m}$), we may further rewrite this
as%
\[
(-1)^{n}\{f_{\alpha(\eta_{n-1})\alpha(\eta_{n})}^{n},\ldots,f_{\alpha(\eta
_{0})\alpha(\eta_{1})}^{1}\}\text{.}%
\]
Since the product in $K$-theory is graded-commutative, we may re-arrange these
terms as%
\[
\mathrm{Eq.}\thinspace \eqref{lcem4}=\sum_{\triangle=(\eta_{n}<\cdots<\eta_{0})}(D\circ\cdots\circ D)_{A_{X}%
(\triangle,\mathcal{O}_{X})}\operatorname*{diag}\nolimits^{\ast}%
\{f_{\alpha(\eta_{0})\alpha(\eta_{1})}^{1},\ldots,f_{\alpha(\eta_{n-1}%
)\alpha(\eta_{n})}^{n}\}\text{,}%
\]
and the factor $(-1)^{n(n-1)/2}$, which had remained from the previous step,
cancels, as this is the sign of the permutation reversing the order of $n$ letters.

Next, we recognize that this expression is exactly of the form $V_{Tate}$ as
in Equation \ref{eqVTate}, so that%
\[
=\sum_{\triangle\in S(X)_{n}}V_{Tate}(\triangle)(\{f_{\alpha(\eta_{0}%
)\alpha(\eta_{1})}^{1},\ldots,f_{\alpha(\eta_{n-1})\alpha(\eta_{n})}^{n}\}).
\]
By Lemma \ref{vtate_vs_valg} this equals%
\[
=\sum_{\triangle\in S(X)_{n}}V_{alg}(\triangle)(\{f_{\alpha(\eta_{0}%
)\alpha(\eta_{1})}^{1},\ldots,f_{\alpha(\eta_{n-1})\alpha(\eta_{n})}^{n}\})
\]
and by Lemma \ref{vger_vs_sum_valg} this equals $V_{Ger}(\{f_{\alpha(\eta
_{0})\alpha(\eta_{1})}^{1},\ldots,f_{\alpha(\eta_{n-1})\alpha(\eta_{n})}%
^{n}\})$. And using the definition of $V_{Ger}$, we unravel this expression as%
\[
=\sum_{\eta_{n},\ldots,\eta_{0}}[\kappa(\eta_{n}):k]\cdot\partial_{\eta_{n}%
}^{\eta_{n-1}}\cdots\partial_{\eta_{1}}^{\eta_{0}}\{f_{\alpha(\eta_{0}%
)\alpha(\eta_{1})}^{1},\ldots,f_{\alpha(\eta_{n-1})\alpha(\eta_{n})}%
^{n}\}\text{.}%
\]
Now use Proposition \ref{prop_intersection_via_boundaryformula} and we conclude that
this equals the intersection multiplicity $L_{1}\cdots L_{n}$.
\epf

\section{\label{sect:TopGroups}Pontryagin duality}
Let us also address a different application of our methods: The classical inspiration for Tate categories are the Tate vector spaces. For a number of applications, especially when one only needs
$n$-Tate objects for $n=1$, one can work with variants of topological $k$-vector spaces instead. And if $k$
happens to be a \emph{finite} field, one can even work with locally compact abelian (LCA) groups.
A discussion in this spirit is given in \cite{highkmoody}, \S 4.1. Indeed, ad\`eles for curves over finite fields are usually considered in the category of LCA groups and Weil's degree formula from the Introduction would classically be phrased in this context. We briefly discuss this:
\medskip

Let $\mathsf{HAb}$ be the category of Hausdorff topological abelian groups.
Morphisms are continuous group homomorphisms. Let $\mathsf{LCA}$ be the category of locally compact abelian topological
groups.

Let $\mathbb{T}$ denote the standard torus group (i.e. $U(1)$ with its usual
topology, or equivalently $\mathbb{R}/\mathbb{Z}$ with the quotient topology).
For $G \in \mathsf{HAb}$ the \emph{dual group} is defined as $G^{\vee}:=\operatorname*{Hom}%
(G,\mathbb{T})$, the group of continuous group homomorphisms, equipped with
the compact-open topology. There is a canonical continuous homomorphism
$\eta_{G}:G\rightarrow G^{\vee\vee}$ and $G$ is called \emph{reflexive} if
$\eta_{G}$ is an isomorphism in $\mathsf{HAb}$. The central duality result for LCA groups is the following:

\begin{theorem}
\textsc{(Pontryagin Duality)} All groups in $\mathsf{LCA}$ are reflexive and there is a
canonical exact equivalence of exact categories%
\[
\mathsf{LCA}^{op}\overset{\sim}{\longrightarrow}\mathsf{LCA}\text{.}%
\]
Suppose $\mathsf{LCA}^{op}\overset
{F}{\longrightarrow}\mathsf{LCA}$ is any equivalence of categories. Then $F$ is
naturally equivalent to the Pontryagin Duality functor.
\end{theorem}

The duality is due to Pontryagin and van Kampen. The uniqueness is due to Roeder \cite[Theorem 5]{MR0279233}.

Next, besides the duality provided by Pontryagin's Duality Theorem, one would like to have a tensor product and internal Homs. This is not readily possible:

\begin{remark}\textsc{(The problem with internal Homs)}
Moskowitz established that if $\operatorname*{Hom}(G,H)$ is locally compact for all
choices of $H$ (resp. $G$), then $G$ must be a finitely generated discrete
group (resp. $H$ must be compact without small sub-groups) \cite[Theorem
4.3]{MR0215016}. Thus, one cannot hope to equip the entire category
$\mathsf{LCA}$ with internal Homs.
\end{remark}

In \cite{MR2329311} Hoffmann and Spitzweck present a different approach to equip $\mathsf{LCA}$
with a closed monoidal structure.

\begin{definition}
An LCA group $A$ has

\begin{enumerate}
\item \emph{finite }$\mathbb{Z}$\emph{-rank} if $\operatorname*{Hom}%
(A,\mathbb{R})$ is a finite-dimensional real vector space,

\item \emph{finite }$\mathbb{S}^{1}$\emph{-rank} if $\operatorname*{Hom}%
(\mathbb{R},A)$ is a finite-dimensional real vector space,

\item \emph{finite }$p$\emph{-rank} if multiplication by $p$, $(-)\cdot
p:A\rightarrow A$ is a strict morphism whose kernel and cokernel are finite
abelian groups.
\end{enumerate}

We say that $A$ has \emph{finite ranks} if its $\mathbb{Z}$-rank is finite,
its $\mathbb{S}^{1}$-rank is finite, and for all prime numbers $p$ its
$p$-rank is finite. Let $\mathsf{FLCA}$ be the full sub-category of
$\mathsf{LCA}$ of groups with finite ranks. (see {\cite[Definition 2.5]{MR2329311}})
\end{definition}

The dual group of a finite ranks group will have finite ranks again, so that
Pontryagin Duality restricts to the category $\mathsf{FLCA}$ and its opposite.

\begin{theorem}\textsc{(Hoffmann--Spitzweck)}
\label{thm:hoffmannspitzweck}The additive categories%
\[
\mathsf{FLCA}\subset\mathsf{LCA}\subset\mathsf{HAb}%
\]
are each quasi-abelian. In particular, they are exact categories in a natural
way. The category $\mathsf{FLCA}$ moreover carries a closed symmetric monoidal
structure with a bi-right exact tensor product%
\[
-\otimes-:\mathsf{FLCA}\times\mathsf{FLCA}\longrightarrow\mathsf{FLCA}\text{.}%
\]
Finally, $\mathsf{FLCA}$ is a fully exact sub-category of $\mathsf{LCA}$.
\end{theorem}

\pf
See \cite{MR2329311}, Proposition 1.2 and Corollary 2.10 for proofs that the
categories are quasi-abelian, Proposition 3.14 for the closed monoidal structure.
The internal Hom is shown to be bi-left exact in \S 3 \textit{loc. cit.} and
the bi-right exactness of the tensor product gets deduced from this, Remark
4.3. \textit{loc. cit. }Finally, $\mathsf{FLCA}$ is fully exact in
$\mathsf{LCA}$ by Proposition 2.9 \textit{loc. cit.}
\epf

Using Corollary \ref{cor_induce_selfequivalences} and the closed monoidal structure of
Theorem \ref{thm:hoffmannspitzweck}, we obtain:

\begin{theorem}
\textsc{(Generalized Pontryagin Duality)} For all $n\geq0$, there is a canonical exact
equivalence of exact categories%
\[
\left.  n\text{-}\mathsf{Tate}(\mathsf{LCA})^{op}\right.  \overset{\sim
}{\longrightarrow}\left.  n\text{-}\mathsf{Tate}(\mathsf{LCA})\right.
\qquad
\left.  n\text{-}\mathsf{Tate}(\mathsf{FLCA})^{op}\right.  \overset{\sim
}{\longrightarrow}\left.  n\text{-}\mathsf{Tate}(\mathsf{FLCA})\right.\text{.}
\]
For any $n,m \geq 1$, there exists a bi-exact functor
            \begin{equation*}
                -\rotimes -\colon \nTateb(\mathsf{FLCA})\times m\text{-}\Tateb(\mathsf{FLCA})\to (n+m)\text{-}\Tateb(\mathsf{FLCA}).
            \end{equation*}
\end{theorem}

\begin{remark}
These categories might be the appropriate candidate when one is interested in a category which
contains both $n$-Tate objects as well as the classical ad\`eles of a number field along with its
archimedean places.
\end{remark}

Classical Pontryagin Duality exchanges certain sub-categories, e.g.,%
\[%
\begin{tabular}
[c]{rcl}%
discrete & $\leftrightarrow$ & compact\\
(topological) $p$-torsion & $\leftrightarrow$ & (topological) $p$-torsion\\
compact metrizable & $\leftrightarrow$ & countable\\
compact connected & $\leftrightarrow$ & torsion-free discrete.
\end{tabular}
\]
By Proposition \ref{prop:duals} our extension of an equivalence $\mathcal{C}%
^{op}\overset{\sim}{\rightarrow}\mathcal{C}$ to the Tate category (1)
preserves the property to exchange such sub-categories, and (2) exchanges Ind-
and Pro-objects. Applying this observation to $\mathsf{LCA}$ and the
aforementioned dual pairs, we get a whole panorama of duality assertions, all
from our general principles. For example, in $\mathsf{Tate}(\mathsf{LCA})$ we
will have%
\[%
\begin{tabular}
[c]{rcl}%
pro-discrete & $\leftrightarrow$ & ind-compact\\
pro-(topological) $p$-torsion & $\leftrightarrow$ & ind-(topological)
$p$-torsion\\
Tate-(compact metrizable) & $\leftrightarrow$ & Tate-countable\\
Tate-(compact connected) & $\leftrightarrow$ & Tate-(torsion-free discrete).
\end{tabular}
\]
Of course this list is not exhaustive, and for $n$-Tate categories, we obtain
all analogous variations for $n$ prefixes Ind-, Pro- or Tate.

\begin{remark}
Barr has also constructed complete and co-complete categories containing
$\mathsf{LCA}$ and all of whose groups are reflexive \cite{MR0578533}. This
might be the largest category which extends Pontryagin Duality within the
context of topological groups.
A classification of the reflexive objects in $\mathsf{HAb}$ appears to be
very complicated, see Hern\'{a}ndez' Theorem \cite[Theorem 3]{MR1869695}.
\end{remark}

We briefly comment on the question to what extent the category $\mathsf{LCA}$ might itself have full sub-categories resembling full sub-categories of Tate object categories. For example, parts of the category $\mathsf{LCA}$ have the shape of Ind- and Pro-objects themselves. This
was first isolated by Roeder:

\begin{proposition}\label{prop_indproinlca}
Let $\mathsf{LCA}_{comp}$ denote the full sub-category of compact LCA
groups, and $\mathsf{LCA}_{disc}$ the full sub-category of discrete groups. Moreover,

\begin{enumerate}
\item let $C_{0}$ be the full sub-category of $\mathsf{LCA}$ of groups of the
shape $\mathbb{T}^{n}\oplus($finite abelian$)$,

\item let $A_{0}$ be the full sub-category of $\mathsf{LCA}$ of groups of the
shape $\mathbb{Z}^{n}\oplus($finite abelian$)$,
\end{enumerate}

Then there are equivalences of categories%
\[
\mathsf{LCA}_{comp}\overset{\sim}{\longrightarrow}\mathsf{Pro}^{a}%
(C_{0})\qquad\text{and}\qquad\mathsf{LCA}_{disc}\overset{\sim}{\longrightarrow
}\mathsf{Ind}^{a}(A_{0})\text{.}%
\]

\end{proposition}

This is \cite[Prop. 5, Prop. 7]{MR0360918}, along with an inspection of
Roeder's definition of the Ind- and Pro-category, which turns out to be
compatible with ours.

The categories $C_{0}$ and $A_{0}$ are Pontryagin duals. Many variations of
this theme are possible, e.g. cardinality constraints.
Thus, if $X_{0}$ denotes the full sub-category of $\mathsf{LCA}$ of groups of
the shape $\mathbb{T}^{n}\oplus\mathbb{Z}^{m}\oplus($finite abelian$)$, then $\mathsf{Tate}^{el}(X_{0})$
is an exact category containing full sub-categories equivalent to copies of
$\mathsf{LCA}_{comp}$ (contained in the Pro-objects) and $\mathsf{LCA}_{disc}$
(contained in the Ind-objects).

It is natural to ask whether the underlying Ind- and\ Pro-objects can be
realized as topological groups. The category $\mathsf{LCA}$ is neither
complete, nor cocomplete, so we cannot carry out Ind- or Pro-limits in
$\mathsf{LCA}$ itself. However, this does not rule out the possibility to work
in a larger category of reflexive topological groups. Indeed, Kaplan showed
that Ind- and Pro-limits of LCA groups, indexed over $\mathbb{N}$, exist in
$\mathsf{HAb}$ and the dual group of such a Pro-limit is the corresponding
Ind-limit of the dual groups \cite[\S 5]{MR0049906}.

\begin{proposition}
There is a natural transformation from the duality functor of Proposition \ref{prop:duals} to dualization on $\mathsf{HAb}$,%
\[
\xymatrix{
{\mathsf{Ind}_{\aleph_{0}}^{a}(\mathsf{LCA})^{op}} \ar[r]^{\sim} \ar[d] & {\mathsf{Pro}_{\aleph_{0}}^{a}(\mathsf{LCA})} \ar[d] \\
\mathsf{HAb}^{op} \ar[r]_{G \mapsto G^{\vee }} & \mathsf{HAb}\text{,}
}
\]
where each downward arrow refers to carrying out the Ind- resp. Pro-limit. If we restrict the bottom row to the essential images of the top row, this diagram extends to a natural equivalence of duality functors.
\end{proposition}

\pf
The category $\mathsf{LCA}$ is quasi-abelian and thus an idempotent complete
exact category. Pontryagin Duality induces an equivalence $\mathsf{LCA}%
^{op}\overset{\sim}{\longrightarrow}\mathsf{LCA}$ and so Proposition
\ref{prop:duals} applies. Now apply Kaplan's duality theorem \cite[\S 5,
Duality Theorem]{MR0049906}.
\epf

\appendix
\section{Construction of a higher Haar-type torsor}

There is a recurring dream in the literature, hoping to establish a good
analogue of harmonic analysis for the ad\`{e}les of a scheme. Unless the
scheme is one-dimensional and has only finite fields as residue fields, the
underlying topological space of the ad\`{e}les fails to be locally compact for
any reasonable choice of a topology on it. This rules out the existence of a
Haar measure. New ideas are needed.

However, whenever a Haar measure exists, it pins down a torsor, encoding the
fact that the Haar measure is only unique up to multiplication with a positive
scalar. We will now give a fairly general construction of the higher
generalization of this torsor, following the idea of \cite{KapranovSemiInfinite} $-$ albeit without addressing the construction
of any form of a generalized measure from which it could stem.\medskip

The papers \cite{MR2473773}, \cite{MR2866188} and \cite{MR3220634} contain interesting
constructions in this respect. In the case of varieties over a field $k$, one
can work with the $C_{n}$-categories of loc. cit., which are a form of Tate
categories. However, when working with number fields or arithmetic surfaces,
one really needs local field factors of the shape $\mathbb{R}$, $\mathbb{C}$,
which exist in $\mathsf{LCA}$, but not in a category like $\mathsf{Tate}%
(\mathbb{F}_{q})$. As a solution, the paper \cite{MR2866188} introduces
categories \textquotedblleft$C_{n}^{\operatorname*{ar}}$\textquotedblright%
\ for $n=0,1,2$. The category $\mathsf{LCA(2)}$ of Zhu \cite{MR3220634} and Liu--Zhu \cite{lca2} is another related candidate. Let $\mathsf{FLCA}_{\aleph}$ be the full sub-category of
$\mathsf{FLCA}$ of second countable groups. This is a quasi-abelian category
by the method of \cite{MR2329311}, and moreover closed under Pontryagin Duality. We consider the category%
\[
\mathcal{C}_{n}^{\operatorname*{Ar}}:=\left.  n\text{-}\mathsf{Tate}\right.
(\mathsf{FLCA}_{\aleph})\text{,}%
\]
which is a quite natural variation/generalization of the categories of
\cite{MR2866188} (note however that the indexing by $n$ has incompatible
meanings! Our $\mathcal{C}_{0}^{\operatorname*{Ar}}$ is philosophically closer
to their $C_{1}^{\operatorname*{ar}}$ than $C_{0}^{\operatorname*{ar}}$). As
an advantage, by its very construction as a Tate category, all the tools of
the paper \cite{MR3510209} are available for $\mathcal{C}_{n}%
^{\operatorname*{Ar}}$. The approach of \cite{MR2866188} sets up the categories
$C_{0}^{\operatorname*{ar}},C_{1}^{\operatorname*{ar}},C_{2}%
^{\operatorname*{ar}}$ by several individual constructions. In particular, we automatically have a
notion of lattices for objects in our $\mathcal{C}_{n}^{\operatorname*{Ar}}$,
we know that any pairs of lattices have a common over- and under-lattice, the
categories are exact and using the tools of the present article, we get a
normally ordered tensor product on the flat Tate objects.

Moreover, we can compute the $K$-theory in terms of the $K$-theory of
$\mathsf{FLCA}_{\aleph}$ using Saito's delooping theorem \cite{MR3317759}.
Unfortunately, the $K$-theory of $\mathsf{FLCA}_{\aleph}$ is not known at the
moment. If $\mathsf{LCA}_{\mathbb{\aleph}}$ denotes the full exact
sub-category of $\mathsf{LCA}$ of second countable groups, Clausen
\cite{clausen} has computed the entire $K$-theory spectrum $K(\mathsf{LCA}%
_{\mathbb{\aleph}})$. Most notably,%
\[
K_{0}(\mathsf{LCA}_{\mathbb{\aleph}})=0\qquad\text{and}\qquad K_{1}%
(\mathsf{LCA}_{\mathbb{\aleph}})\cong\mathbb{R}_{>0}^{\times}\text{,}%
\]
and the latter isomorphism has the following explicit description: If
$G\in\mathsf{LCA}_{\mathbb{\aleph}}$ is an LCA\ group, pick a Haar measure
$\mu$ on it. Now, every automorphism $\gamma$ of $G$ canonically determines an
element in $K_{1}(\mathsf{LCA}_{\mathbb{\aleph}})$, giving a map%
\begin{equation}
C:\operatorname{Aut}_{\mathsf{LCA}_{\mathbb{\aleph}}}(G)\longrightarrow
K_{1}(\mathsf{LCA}_{\mathbb{\aleph}})\overset{\sim}{\longrightarrow}%
\mathbb{R}_{>0}^{\times}\label{gk_2}%
\end{equation}
by Clausen's computation. At the same time, $\mu\circ\gamma$ is again a
bi-invariant measure on $G$, and by the uniqueness of Haar measures up to
rescaling by a positive constant, there is a unique $c_{\gamma}\in
\mathbb{R}_{>0}^{\times}$ such that $(\mu\circ\gamma)(U)=c_{\gamma}\cdot\mu(U)$
holds for all measurable sets $U\subseteq G$. This construction defines a
further map $c:\operatorname{Aut}_{\mathsf{LCA}_{\mathbb{\aleph}}%
}(G)\rightarrow\mathbb{R}_{>0}^{\times}$. According to Clausen's computation,
both maps agree, i.e. $C=c$. Following Weil, $c_{\gamma}$ is called the
\emph{module} of an automorphism, see \cite[Ch. I, \S 2]{MR0427267}.

We may now generalize this to construct a type of higher Haar-type torsor for
objects in the categories $\mathcal{C}_{n}^{\operatorname*{Ar}}$. Saito's
delooping theorem provides a canonical equivalence%
\begin{equation}
K(\mathcal{C}_{n}^{\operatorname*{Ar}})\overset{\sim}{\longrightarrow}%
B^{n}K(\mathsf{FLCA}_{\mathbb{\aleph}})\label{gk_3}%
\end{equation}
and since $\mathsf{FLCA}_{\mathbb{\aleph}}\hookrightarrow\mathsf{LCA}%
_{\mathbb{\aleph}}$ is a fully exact sub-category, the functor of inclusion is
exact and thus we have an induced map in $K$-theory,
\[
K(\mathsf{FLCA}_{\mathbb{\aleph}})\longrightarrow K(\mathsf{LCA}%
_{\mathbb{\aleph}})\text{.}%
\]
The truncation of the $K$-theory spectrum of $\mathsf{LCA}_{\mathbb{\aleph}}$
to degrees $[0,1]$ yields the Eilenberg-MacLane spectrum of the group
$\mathbb{R}_{>0}^{\times}$ (taken with the discrete topology) in degree $1$. Thus, composing with the truncation
map, the above map induces a map%
\begin{equation}
K(\mathcal{C}_{n}^{\operatorname*{Ar}})\longrightarrow B^{n}K(\mathsf{LCA}%
_{\mathbb{\aleph}})\longrightarrow B^{n}\tau_{\leq1}K(\mathsf{LCA}%
_{\mathbb{\aleph}})\overset{\sim}{\longrightarrow}B^{n+1}\mathbb{R}%
_{>0}^{\times}\text{.}\label{gk1}%
\end{equation}
Suppose $G\in\mathcal{C}_{n}^{\operatorname*{Ar}}$. The generalization of the
above construction with automorphisms is the canonical map%
\[
B\operatorname{Aut}_{\mathcal{C}_{n}^{\operatorname*{Ar}}}(G)\longrightarrow
\Omega^{\infty}K(\mathcal{C}_{n}^{\operatorname*{Ar}})\text{,}%
\]
where the left-hand side is the classifying space of the group
$\operatorname{Aut}_{\mathcal{C}_{n}^{\operatorname*{Ar}}}(G)$, and the
right-hand side the infinite loop space attached to the $K$-theory spectrum of
$\mathcal{C}_{n}^{\operatorname*{Ar}}$, which happens to be connective. In the
special case $n=0$, this induces on the level of the fundamental group
$\pi_{1}$ the map $C:\operatorname{Aut}_{\mathcal{C}_{n}^{\operatorname*{Ar}}%
}(G)\rightarrow\mathbb{R}_{>0}^{\times}$ discussed above. Now we pre-compose
the map of Equation \ref{gk1} with this construction, giving the following:

\begin{definition}
\label{def_HigherHaarTorsor}For every $n\geq0$ and every object $G\in
\mathcal{C}_{n}^{\operatorname*{Ar}}$, there is a canonical map of spaces
\[
B\operatorname{Aut}_{\mathcal{C}_{n}^{\operatorname*{Ar}}}(G)\longrightarrow
B^{n+1}\mathbb{R}_{>0}^{\times}\text{.}%
\]
In particular, this defines a canonical group cohomology class%
\[
H\in H_{\operatorname*{grp}}^{n+1}(\operatorname{Aut}_{\mathcal{C}%
_{n}^{\operatorname*{Ar}}}(G),\mathbb{R}_{>0}^{\times})\text{,}%
\]
which we call the (cohomology class of the) \emph{higher Haar torsor} on $G$.
\end{definition}

To construct $H$, we use the standard fact from homotopy theory that for any
group $A$ and abelian group $Z$, group cohomology has the description%
\[
H_{\operatorname*{grp}}^{n}(A,Z)=\pi_{0}\operatorname*{map}(BA,B^{n}Z)\text{,}%
\]
i.e. group cohomology classes correspond to homotopy classes of maps from $BA$
to $B^{n}Z$. As the above constructions produce such a map, this pins down a
cohomology class.

\begin{remark}
\textsc{(Haar-type torsor)} Having constructed the cohomology class $H$, one can also speak
of the actual higher Haar-type torsor, instead of just the map classifying it. For
this one needs to specify what the word \textquotedblleft higher
torsor\textquotedblright\ should mean. For example, following the approach of
Saito to higher $K$-theory torsors \cite{Saito:2014fk}, one can use the
framework of \cite{MR3423073}. This Haar-type torsor thus fits into the pattern of
Saito's general construction.
\end{remark}

\begin{remark}
Unlike \cite{MR2866188}, we only construct higher Haar-type torsor candidates, but we
make no attempt here to construct some generalized measure theory such that
our Haar torsor would arise from it. Even having the higher Haar torsor
available, this does not clarify in any way what, for example, a generalized
Schwartz--Bruhat function should be.
\end{remark}

\begin{proposition}
\textsc{(Agreement)} For $n=0$, $\mathcal{C}_{0}^{\operatorname*{Ar}}=\mathsf{FLCA}%
_{\mathbb{\aleph}}$ is the category of second countable LCA groups of finite
ranks \`{a} la Hoffmann--Spitzweck \cite{MR2329311}. For every such group $G$,
the Haar torsor of Definition \ref{def_HigherHaarTorsor} is a classical
$\mathbb{R}_{>0}^{\times}$-torsor and%
\[
H\in H^{1}(\operatorname{Aut}(G),\mathbb{R}_{>0}^{\times})
\]
is just the map $c:\operatorname{Aut}(G) \rightarrow\mathbb{R}_{>0}^{\times}$,
\begin{align*}
\gamma & \longmapsto\frac{\mu(\gamma U)}{\mu(U)}\text{,}%
\end{align*}
where $\mu$ is any Haar measure on $G$, and $U$ any measurable sub-group of
$G$ of positive measure.
\end{proposition}

\pf
The property $\mathcal{C}_{0}^{\operatorname*{Ar}}=\mathsf{FLCA}%
_{\mathbb{\aleph}}$ holds by construction. The class $H$ agrees with $C$ of
Equation \ref{gk_2}, and by\ Clausen's computation in \cite{clausen}, $C=c$,
giving the claim. As a result, the torsor associated to $H$ by the classical
theory of torsors (or compatibly \cite{MR3423073}), is isomorphic to the
torsor of Haar measures coming from classical harmonic analysis.
\epf

\begin{remark}
A description of Saito's equivalence in Equation \ref{gk_3} by an explicit
simplicial model based on lattices in Tate objects is given in
\cite{IndexMapAlgKTheory}. Thus, if one wants to study the higher Haar torsors
in an explicit fashion, this would be a first step.
\end{remark}

\begin{example}
As a concrete example, we get a Haar measure analogue of the tame symbol. Let
us spell out the details: For $n=1$, every object $G\in\mathcal{C}%
_{1}^{\operatorname*{Ar}}$ comes with a canonical central extension%
\[
1\longrightarrow\mathbb{R}_{>0}^{\times}\longrightarrow\widehat
{\operatorname{Aut}(G)}\longrightarrow\operatorname{Aut}(G)\longrightarrow1
\]
of its automorphism group, as defined by $H\in H^{2}(\operatorname{Aut}%
(G),\mathbb{R}_{>0}^{\times})$. As part of this structure for any two
\emph{commuting} automorphisms $f,g$ of $G$, we get an element%
\[
\left\langle f,g\right\rangle \in\mathbb{R}_{>0}^{\times}\text{.}%
\]
When restricting to an abelian sub-group $A\subseteq\operatorname{Aut}(G)$ for
example, all its elements pairwise commute, and then $\left\langle
-,-\right\rangle $ becomes a bi-linear pairing%
\[
\left\langle -,-\right\rangle :A\times A\longrightarrow\mathbb{R}_{>0}%
^{\times}\text{.}%
\]
If we read the symbols $F:=\mathbb{R},\mathbb{C},\mathbb{Q}_{p}$ as referring
to their corresponding objects in $\mathsf{FLCA}$, we can consider the
following object in $\mathcal{C}_{1}^{\operatorname*{Ar}}$,%
\[
F((t)):=\underset{i}{\operatorname*{colim}}\underset{j}{\lim}\left.
t^{-i}F[t]/t^{j}\right.  \text{.}%
\]
Clearly $F((t))$ (now viewed as a ring) acts by multiplication on $F((t))$,
giving the abelian subgroup $F((t))^{\times}$ inside the automorphism group of
$F((t))\in\mathcal{C}_{1}^{\operatorname*{Ar}}$. Then we get the pairing%
\begin{equation}
F((t))^{\times}\times F((t))^{\times}   \longrightarrow\mathbb{R}%
_{>0}^{\times}\qquad\qquad
(f,g) \longmapsto\left\vert \frac{f^{v(g)}}{g^{v(f)}}(0)\right\vert
\text{,}%
\end{equation}
where $v(-)$ refers to the $t$-valuation, so that $f^{v(g)}/g^{v(f)}$ is a
formal power series with non-zero constant coefficient in $F$, and $\left\vert
\cdot\right\vert $ refers to the real absolute value for $F=\mathbb{R}$, the
square of the complex absolute value $\left\vert \cdot\right\vert ^{2}$ for
$F=\mathbb{C}$, or the $p$-adic absolute value for $F=\mathbb{Q}_{p}$.
Moreover, if $\mathbb{A}$ denotes the classical ad\`{e}les of a number field
or a curve over a finite field, viewed as an LCA\ group, we get%
\begin{equation}
\mathbb{A}((t))^{\times}\times\mathbb{A}((t))^{\times} \longrightarrow
\mathbb{R}_{>0}^{\times}\qquad\qquad
(f,g) \longmapsto\prod_{w}\left\vert \frac{f^{v(g)}}{g^{v(f)}%
}(0)\right\vert \text{,}%
\end{equation}
where $w$ runs over the places of the number field/function field, and all but
finitely many factors in this product are equal to one. To prove the
assertions, use the following trick: In the cases $F:=\mathbb{R}%
,\mathbb{C},\mathbb{Q}_{p}$ use that there is an exact functor from
finite-dimensional $F$-vector spaces to $\mathsf{FLCA}_{\aleph}$,%
\[
\mathsf{Vect}_{f}(F)\longrightarrow\mathsf{FLCA}_{\aleph}\text{,}%
\]
which sends a finite-dimensional $F$-vector space to its additive group,
equipped with the topology coming from $F$. In the cases at hand, this is
always a second countable LCA\ group. This functor is easily checked to be
exact. Thus, the functor%
\[
\left.  n\text{-}\mathsf{Tate}\right.  (\mathsf{Vect}_{f}(F))\longrightarrow
\left.  n\text{-}\mathsf{Tate}\right.  (\mathsf{FLCA}_{\aleph})=\mathcal{C}%
_{n}^{\operatorname*{Ar}}%
\]
is exact. This has a great advantage for our computation: If we know the
result for $\left.  n\text{-}\mathsf{Tate}\right.  (\mathsf{Vect}_{f}(F))$, we
can simply make the computation there and then map it to $\mathsf{FLCA}%
_{\aleph}$. As is shown in \cite{bgwRelativeTateObjects} or
\cite{IndexMapAlgKTheory}, for $\left.  n\text{-}\mathsf{Tate}\right.
(\mathsf{Vect}_{f}(F))$, the delooping map is compatible with the boundary map
in $K$-theory of the exact sequence of exact categories%
\[
\operatorname*{Coh}\nolimits_{\overline{\{(t)\}}}F[[t]]\longrightarrow
\operatorname*{Coh}F[[t]]\longrightarrow\operatorname*{Coh}F((t))
\]
and on the level relevant for our computation, this is just the ordinary tame
symbol
\[
(-1)^{v(f)v(g)}f^{v(g)}g^{-v(f)}(0)\in F^{\times}\text{,}
\]
which lives in
$K_{1}(F)\cong K_{1}(\operatorname*{Coh}\nolimits_{\overline{\{(t)\}}}%
F[[t]])$. Once we have this element in $F^{\times}$, the induced map
$K_{1}(F)\longrightarrow K_{1}(\mathsf{LCA}_{\aleph})$ is%
\begin{align*}
F^{\times}  & \longrightarrow\mathbb{R}_{>0}^{\times}\\
\alpha & \longmapsto\left\vert \alpha\right\vert _{\mathbb{R}}\text{ (resp.
}\left\vert \alpha\right\vert _{\mathbb{C}}^{2}\text{, resp. }\left\vert
\alpha\right\vert _{p}\text{).}%
\end{align*}
This observation is due to Andr\'{e} Weil, and in a sense lies at the roots of
the subject: For all the fields which occur in the classical ad\`{e}les,
namely locally compact topological fields, one can uniquely reconstruct their
natural notions of absolute value (resp. its square) by studying the module of automorphisms with
respect to any Haar measure. We refer to \cite[Chapter 1]{MR0427267}, which is
the ultimate treatment using this perspective. This proves our claim for
$F:=\mathbb{R},\mathbb{C},\mathbb{Q}_{p}$. We leave the case of the ad\`{e}les
to the reader.
\end{example}

\begin{example}
Following Clausen, we also get vanishing statements in the context of the previous example. If we let $F:=\mathbb{Z}$ or $\mathbb{T}$ as an object of $\mathsf{FLCA}$, the
corresponding Haar central extensions on%
\[
\mathbb{Z}((t))\qquad\text{resp.}\qquad\mathbb{T}((t))
\]
are trivial, i.e. they split. This can be seen as follows: Just as the above computation factored the classifying map to
$\mathsf{LCA}_{\aleph}$ over $\mathsf{Vect}_{f}(F)$, we may now factor over
$\mathsf{LCA}_{disc,\aleph}$ resp. $\mathsf{LCA}_{comp,\aleph}$, and the
$K$-theory of these categories is contractible. Indeed, these are an Ind- and
a Pro-category themselves, by a mild variation of Proposition \ref{prop_indproinlca}.
\end{example}

\bibliographystyle{amsalpha}
\bibliography{ollinewbib,master}

\def\cprime{$'$} \def\polhk#1{\setbox0=\hbox{#1}{\ooalign{\hidewidth
  \lower1.5ex\hbox{`}\hidewidth\crcr\unhbox0}}} \def\cprime{$'$}
  \def\cprime{$'$} \def\cprime{$'$} \def\cprime{$'$}
\providecommand{\bysame}{\leavevmode\hbox to3em{\hrulefill}\thinspace}
\providecommand{\MR}{\relax\ifhmode\unskip\space\fi MR }
\providecommand{\MRhref}[2]{%
  \href{http://www.ams.org/mathscinet-getitem?mr=#1}{#2}
}
\providecommand{\href}[2]{#2}
\begin{thebibliography}{BGW16b}

\bibitem[Bar77]{MR0578533}
M.~Barr, \emph{A closed category of reflexive topological abelian groups},
  Cahiers Topologie G\'eom. Diff\'erentielle \textbf{18} (1977), no.~3,
  221--248. \MR{0578533}

\bibitem[BD04]{BeD:04}
A.~Beilinson and V.~Drinfeld, \emph{Chiral {A}lgebras}, American Mathematical
  Society, 2004.

\bibitem[Bei80]{MR565095}
A.~Beilinson, \emph{Residues and ad\`eles}, Funktsional. Anal. i Prilozhen.
  \textbf{14} (1980), no.~1, 44--45. \MR{565095 (81f:14010)}

\bibitem[Bei08]{Bei:08}
A.~Beilinson, \emph{Remarks on {T}opological {A}lgebras}, Mosc. Math. J.
  \textbf{8} (2008), 1--20.

\bibitem[BGW16a]{MR3536437}
O.~Braunling, M.~Groechenig, and J.~Wolfson, \emph{Geometric and analytic
  structures on the higher ad\`eles}, Res. Math. Sci. \textbf{3} (2016), Paper
  No. 22, 56. \MR{3536437}

\bibitem[BGW16b]{bgwTateModule}
\bysame, \emph{Operator ideals in {T}ate objects}, Math. Res. Lett. \textbf{23}
  (2016), no.~6, 1565--1631. \MR{3621099}

\bibitem[BGW16c]{MR3510209}
\bysame, \emph{Tate objects in exact categories}, Mosc. Math. J. \textbf{16}
  (2016), no.~3, 433--504, With an appendix by Jan
  {\v{S}}{\v{t}}ov{\'{\i}}{\v{c}}ek and Jan Trlifaj. \MR{3510209}

\bibitem[BGW17]{bgwRelativeTateObjects}
\bysame, \emph{Relative {T}ate objects and boundary maps in the {$K$}-theory of
  coherent sheaves}, Homology Homotopy Appl. \textbf{19} (2017), no.~1,
  341--369. \MR{3652927}

\bibitem[BGW18]{IndexMapAlgKTheory}
\bysame, \emph{The index map in algebraic ${K}$-theory}, Selecta Math.,
  https://doi.org/10.1007/s00029-017-0364-0 (2018).

\bibitem[Bra13]{MR3104562}
O.~Braunling, \emph{Bad intersections and constructive aspects of the
  {B}loch-{Q}uillen formula}, New York J. Math. \textbf{19} (2013), 545--564.
  \MR{3104562}

\bibitem[Cla17]{clausen}
D.~Clausen, \emph{A {K}-theoretic approach to {A}rtin maps}, arXiv:1703.07842
  [math.KT] (2017).

\bibitem[FHK]{highkmoody}
G.~Faonte, B.~Hennion, and M.~Kapranov, \emph{Higher {K}ac--{M}oody algebras
  and moduli spaces of {G}-bundles}, arXiv:1701.01368 [math.AG].

\bibitem[Her01]{MR1869695}
S.~Hern{\'a}ndez, \emph{Pontryagin duality for topological abelian groups},
  Math. Z. \textbf{238} (2001), no.~3, 493--503. \MR{1869695}

\bibitem[HS07]{MR2329311}
N.~Hoffmann and M.~Spitzweck, \emph{Homological algebra with locally compact
  abelian groups}, Adv. Math. \textbf{212} (2007), no.~2, 504--524. \MR{2329311
  (2009d:22006)}

\bibitem[Kap50]{MR0049906}
S.~Kaplan, \emph{Extensions of the {P}ontrjagin duality. {II}. {D}irect and
  inverse sequences}, Duke Math. J. \textbf{17} (1950), 419--435. \MR{0049906}

\bibitem[Kap01]{KapranovSemiInfinite}
M.~M. Kapranov, \emph{{Semi-infinite symmetric powers}}, {arXiv:math/0107089v1
  [math.QA]}, 2001.

\bibitem[Ker10]{MR2551760}
M.~Kerz, \emph{Milnor {$K$}-theory of local rings with finite residue fields},
  J. Algebraic Geom. \textbf{19} (2010), no.~1, 173--191. \MR{2551760
  (2010j:19006)}

\bibitem[LZ17]{lca2}
D.~Liu and Y.~Zhu, \emph{{LCA(2), Weil index, and product formula}},
  {arXiv:1708.09205 [math.RT]} (2017).

\bibitem[Mos67]{MR0215016}
M.~Moskowitz, \emph{Homological algebra in locally compact abelian groups},
  Trans. Amer. Math. Soc. \textbf{127} (1967), 361--404. \MR{0215016}

\bibitem[NSS15]{MR3423073}
T.~Nikolaus, U.~Schreiber, and D.~Stevenson, \emph{Principal
  {$\infty$}-bundles: general theory}, J. Homotopy Relat. Struct. \textbf{10}
  (2015), no.~4, 749--801. \MR{3423073}

\bibitem[OP08]{MR2473773}
D.~V. Osipov and A.~N. Parshin, \emph{Harmonic analysis on local fields and
  adelic spaces. {I}}, Izv. Ross. Akad. Nauk Ser. Mat. \textbf{72} (2008),
  no.~5, 77--140. \MR{2473773 (2010a:11225)}

\bibitem[OP11]{MR2866188}
\bysame, \emph{Harmonic analysis on local fields and adelic spaces. {II}}, Izv.
  Ross. Akad. Nauk Ser. Mat. \textbf{75} (2011), no.~4, 91--164. \MR{2866188
  (2012h:11165)}

\bibitem[Par83]{MR697316}
A.~N. Parshin, \emph{Chern classes, ad\`eles and {$L$}-functions}, J. Reine
  Angew. Math. \textbf{341} (1983), 174--192. \MR{697316 (85c:14015)}

\bibitem[Roe71]{MR0279233}
D.~Roeder, \emph{Functorial characterizations of {P}ontryagin duality}, Trans.
  Amer. Math. Soc. \textbf{154} (1971), 151--175. \MR{0279233}

\bibitem[Roe74]{MR0360918}
\bysame, \emph{Category theory applied to {P}ontryagin duality}, Pacific J.
  Math. \textbf{52} (1974), 519--527. \MR{0360918}

\bibitem[Ros96]{MR1418952}
M.~Rost, \emph{Chow groups with coefficients}, Doc. Math. \textbf{1} (1996),
  No. 16, 319--393 (electronic). \MR{1418952 (98a:14006)}

\bibitem[Sai14]{Saito:2014fk}
S.~Saito, \emph{Higher {T}ate central extensions via {K}-theory and
  infinity-topos theory}, arXiv:1405.0923, 05 2014.

\bibitem[Sai15]{MR3317759}
\bysame, \emph{On {P}revidi's delooping conjecture for {$K$}-theory}, Algebra
  Number Theory \textbf{9} (2015), no.~1, 1--11. \MR{3317759}

\bibitem[Sch10]{MR2600285}
M.~Schlichting, \emph{Hermitian {$K$}-theory of exact categories}, J. K-Theory
  \textbf{5} (2010), no.~1, 105--165. \MR{2600285}

\bibitem[Wal85]{MR802796}
F.~Waldhausen, \emph{Algebraic {$K$}-theory of spaces}, Algebraic and geometric
  topology ({N}ew {B}runswick, {N}.{J}., 1983), Lecture Notes in Math., vol.
  1126, Springer, Berlin, 1985, pp.~318--419. \MR{802796}

\bibitem[Wei74]{MR0427267}
A.~Weil, \emph{Basic number theory}, third ed., Springer-Verlag, New
  York-Berlin, 1974, Die Grundlehren der Mathematischen Wissenschaften, Band
  144. \MR{0427267}

\bibitem[Yek92]{MR1213064}
A.~Yekutieli, \emph{An explicit construction of the {G}rothendieck residue
  complex}, Ast\'erisque (1992), no.~208, 127, With an appendix by Pramathanath
  Sastry. \MR{1213064 (94e:14026)}

\bibitem[Zhu14]{MR3220634}
Y.~Zhu, \emph{Weil representations and theta functionals on surfaces},
  Perspectives in representation theory, Contemp. Math., vol. 610, Amer. Math.
  Soc., Providence, RI, 2014, pp.~353--370. \MR{3220634}

\end{thebibliography}

\end{document}